\documentclass[12pt]{elsarticle}
\usepackage{graphicx}
\usepackage{subfigure}
\usepackage{booktabs}   
\usepackage{hhline}     
\usepackage{epstopdf} 

\usepackage{booktabs}
\usepackage{color}
\usepackage{latexsym}
\usepackage{setspace}
\usepackage{amsmath}
\usepackage{amssymb}
\usepackage{latexsym}
\usepackage{amsmath}
\usepackage{amssymb}
\usepackage[numbers]{natbib}
\usepackage{hyperref}
\hypersetup{
	colorlinks=true,
	linkcolor=blue,
	filecolor=blue,
	urlcolor=blue,
}
\usepackage{subfigure}
\usepackage{graphics}
\usepackage{graphicx}
\usepackage{multirow}
\makeatletter
\def\ps@pprintTitle{%
	\let\@oddhead\@empty
	\let\@evenhead\@empty
	\let\@oddfoot\@empty
	\let\@evenfoot\@oddfoot
}
\makeatother
\makeatletter \oddsidemargin  -.5cm \evensidemargin -.5cm \textwidth
17.5cm \topmargin -.65in \textheight 24cm

\newtheorem{defn}{Definition}[section]
\newcommand{\bd}{\begin{definition}}
	\newcommand{\ed}{\end{definition}}
\newcommand{\br}{\begin{remark}}
	\newcommand{\er}{\end{remark}}
\newcommand{\bea}{\begin{eqnarray}}
	\newcommand{\eea}{\end{eqnarray}}
\newcommand{\beann}{\begin{eqnarray*}}
	\newcommand{\eeann}{\end{eqnarray*}}
\newtheorem{theorem}{Theorem}[section]

\newtheorem{lemma}[theorem]{Lemma}

\newtheorem{corollary}[theorem]{Corollary}
\newtheorem{remark}{Remark}[section]
\newtheorem{example}{Example}[section]

\numberwithin{equation}{section} \numberwithin{equation}{section}
\title{Improved estimation of positive powers of scale parameters of exponential distributions under a prior information}
\author{Somnath Mondal\footnote
	{\baselineskip=10pt
		~somnathmondal@iitbhilai.ac.in; somnathmondal1498@gmail.com}
	\\
	Department of Mathematics\\
	Indian Institute of Technology Bhilai, Bhilai, India-491002}
\begin{document}
	\begin{frontmatter}
		\date{}
		\begin{abstract}
			Estimating unknown parameters subject to prior constraints is important in statistical inference, particularly in fields such as reliability analysis, survival studies, and engineering, where prior structural information about the parameters is often available. Incorporating such prior information makes the analysis more realistic and usually yields better estimates than methods that ignore such information. In this article, we consider the problem of estimating the positive power of the scale parameter of a two-shifted exponential population under a prior ordering constraint on scale parameters. We derive sufficient conditions under which equivariant estimators are shown to dominate others under scale-invariant strictly convex loss functions. In addition, we derived various estimators that dominate the best affine equivariant estimators (BAEE). Moreover, we derive a smooth estimator which dominates the BAEE using an integrated approach, and we further show that it is a generalized Bayes estimator under a non-informative prior. We also provide an improved estimator based on the Pitman closeness criterion. An extensive simulation study has been done for computational purposes. Finally, we provided real examples to implement the results.\\			
			
		\end{abstract}

	\end{frontmatter}
		\section{Introduction} 
	The problem of estimating parameter under some structural constraints has received a significant attention in statistical inference due to its wide applicability in reliability theory, life testing experiments and biomedical studies. In many practical situations parameters are not independent but have certain order relation. For example, in an industrial process machines operating under better maintenance conditions tends to have longer operating lifetimes than the poorly maintained one this resulting in ordered scale parameters. Among various lifetime models, the exponential distribution plays a fundamental role due to its practical usefulness in modelling failure time data. In this context, the estimation of scale parameters becomes significant specially when the scale parameters in some order relation and the location parameter are unknown with some order restriction. Incorporating the prior knowledge such as an ordering constraint on scale parameters are often yields estimators that outperform then the traditional estimators. Estimation of ordered parameters has extensively studied by various authors. Most of the existing work on estimating ordered parameters when prior knowledge suggest that they follow certain order restrictions has primarily concentrated on the development of maximum likelihood estimators. Some early work in this direction was carried out by \cite{barlow1972statistical} and \cite{robertson1988order}. Later contributions to this direction can be found in \cite{blumenthal1968estimation}, \cite{gupta1992pitman}, \cite{misra1995some}, \cite{vijayasree1995componentwise}, \cite{misra1997estimation}.
	Due to the wide applicability of exponential models in many real-life situations, considerable attention has been given in recent years to the estimation of order-restricted parameters in exponential distributions. \cite{vijayasree1995componentwise} consider the problem of component wise estimation of ordered parameters of several exponential distribution. Using the \cite{brewster1974improving} technique, sufficient conditions are derived to identify when the usual estimators of the location and scale parameters are inadmissible under the mean squared error criterion. Based on these results, improved estimators are proposed. \cite{misra2002smooth} studied the estimation of ordered scale parameters for two gamma distributions under the assumption that the ordering is known a priori. They considered component-wise estimation under a scale-equivariant squared error loss function and proposed smooth estimators that improve upon the best scale equivariant estimators. These estimators were further shown to be generalized Bayes under a non-informative prior. In \cite{jana2015estimation}, the authors studied the estimation of ordered scale parameters for two exponential distributions sharing a common location parameter and showed that the restricted maximum likelihood estimators outperform the usual maximum likelihood estimators in terms of risk. These results were later extended by \cite{jena2017estimating} by considering the problem under Type-II censoring schemes. Under the same model assumption, \cite{kayal2024estimating} recently investigated the estimation of ordered scale parameters for several exponential populations while \cite{jena2025inference} considered the problem of estimation for powers of the scale parameters of two exponential populations under the restriction of a common location parameter. For further works on the estimation of ordered scale parameters under various settings, one may refer to \cite{jana2015estimation}, \cite{patra2021componentwise}, \cite{petropoulos2017estimation}, \cite{kayal2024estimating},  \cite{mondal2024improved}. 
	
	In this article, we study the component wise estimation of positive power of ordered scale parameters of two exponential distribution when location parameters are unknown and unequal under a bowl shaped loss function $L\left(\frac{\delta}{\nu}\right)$, where $\delta$ is an estimator of $\nu$. We consider a loss function $L(t)$ which satisfies the following properties:
	\begin{itemize}
		\item [(A1)] $L(t)$ is convex with $L(1)=0$,
		\item [(A2)] The integral involving $L(t)$ are finite and can be differentiable under the integral sign.
	\end{itemize}
	Let we consider $(X_{i1}, X_{i2}, \dots, X_{ip_i})$ be random samples taken from the $i$th populations $\pi_i$, $i=1,2$. The probability density function of the $i$th population $\pi_i$ is given by
	\begin{align}
		g_i(x;\mu_i,\sigma_i)=\begin{cases}
			\frac{1}{\sigma_i}\mbox{exp}\left(-\frac{x-\mu_i}{\sigma_i}\right),~\ \mbox{if} ~x>\mu_i,\\
			0,~~~~~~~~~~~~~~~~~~~~\mbox{otherwise}
		\end{cases}
	\end{align}
	where $-\infty<\mu_i<\infty$ and $\sigma_i>0$. We assume the parameters $\sigma_i$'s and $\mu_i$'s, $i=1,2$ are unknown and $\sigma_1\le \sigma_2$. We now reduce the the sample without loss any important information about the parameter. To do that we have the complete sufficient statistics are $X_1,X_2, S_1,S_2$. We denote $X=(X_1,X_2)$ and $S=(S_1,S_2)$ where for $i=1,2$ we have
	\begin{align*}
		X_i=\min_{1\leq j \leq p_i} X_{ij} \sim E\left(\mu_i,\sigma_i/p_i\right), ~~~ S_i=\sum_{j=1}^{p_i}\left(X_{ij}-X_i\right)\sim Gamma\left(p_i-1,\sigma_i\right).
	\end{align*} Note that $(X_1,X_2,S_1,S_2)$ are independently distributed. Our goal is to estimate $\sigma_i^k$ for $k\in \mathbb{R}$ under a general scale invariant loss function $L(t)$ which satisfies the condition (A1) and (A2).

	Consider the affine group of transformations $\mathcal{G}=\left\{ g_{a_i,b_i}(x)=a_ix_{ij}+b_i\right\}$. Under this group of transformation we get the form of affine equivariant estimators which can be obtained as
	\begin{align}\label{aee}
		\mathcal{D}_0=\left\{\delta_d;\ \ \delta_{1d}(X_i,S)=dS_i^k \right\}
	\end{align}
	where $d$ is a real constant. The following theorem gives the best estimator from the class $\mathcal{D}_0$, we say it BAEE.
	
	\begin{lemma}
		Under a scale invariant loss function given by $L(t)$ , the best affine equivariant estimator (BAEE) of $\sigma_i$ is $\delta_{0i}(X_i,S)=d_{0i}S_i^k$, where $d_{0i}$ is the unique solution of the equation
		\begin{equation}
			E\left( L^\prime(d_{0i}V_i^k)V_i^k\right)=0
		\end{equation} where $V_i=\frac{S_i}{\sigma_i}\sim Gamma \left(p_i-1,1\right)$
	\end{lemma}
	\begin{example}\label{1.1}
		For $i=1,2$
		\begin{enumerate}
			\item[(i)] Let the quadratic loss function $L_1(t)=(t-1)^2$. The BAEE of $\sigma_i^k$ is $\delta^1_{0i}=\frac{\Gamma(p_i+k-1)}{\Gamma(p_i+2k-1)}S_i^k$.
			
			\item[(ii)] Consider the entropy loss function $L_2(t)=t-\ln t-1$, the BAEE of $\sigma_i^k$ is $\delta^2_{0i}=\frac{\Gamma(p_i-1)}{\Gamma(p_i+k-1)}S_i^k$.
			
			\item [(iii)] For the symmetric loss function $L_3(t)=t+1/t-2$, the BAEE of $\sigma_i^k$ is $\delta^3_{0i}=\sqrt{\frac{\Gamma(p_i-k-1)}{\Gamma(p_i+k-1)}}S_i^k$
			
			\item[(iv)] For linex loss function $L_4(t)=e^{\alpha(t-1)}-\alpha(t-1) -1$, $\alpha \in \mathbb{R}-\{0\}$, the BAEE of $\sigma_i^k$ is $\delta^4_{0i}=d_{0i}S_i^k$, where $d_{0i}$ is the unique solution of the equation 
			\begin{equation*}
				\int_{0}^{\infty}v_i^{p_i+k-2}e^{\alpha d_{0i} v_i^k - v_i} dv_i = e^\alpha \int_{0}^{\infty} v_i^{p_i+k-2} e^{-v_i} dv_i
			\end{equation*}

		\end{enumerate}
	\end{example}
	\begin{remark}
		The UMVUE of $\sigma_i^k$ is $\delta_{Ui}=\frac{\Gamma(p_i-1)}{\Gamma(p_i+k-1)}S_i^k$ and it belongs to the class $\mathcal{D}_0$
	\end{remark}
	The main contributions of this article are summarized as follows. In Section \ref{se2}, we consider the problem of estimating $\sigma_1^k$ under the constraint $\sigma_1\leq \sigma_2$ and derive estimators that dominates the BAEE. Furthermore, we obtain a smooth improved estimator that dominates the BAEE. Also a generalized Bayes estimator is obtained under a non-informative prior. In Section \ref{se3}, analogous results are established for the estimation of $\sigma_2^k$. In Section \ref{se4}, we derive improved estimator of $\sigma_i^k$ under generalized Pitman closeness criterion. An extensive simulation study is conducted in Section \ref{se5} to compare the risk performance of the proposed estimators. Finally, Section \ref{se6} provides an application of the proposed methods to a real life data set.
	
	\section{Improved estimation of $\sigma_1^k$}\label{se2}
	In this section, we consider the problem of obtaining an improved estimator of $\sigma_1^k$ under the constraint $\sigma_1\leq \sigma_2$.	We consider the class of estimator of the form 
	\begin{align*}
		\mathcal{D}_1=\left\{ \delta_{\varphi_1}:\ \  \delta_{\varphi_1}\left(X,S\right)=\varphi_1(T)S_1^k\right\}
	\end{align*}
	where $T=\frac{S_2}{S_1}$ and $\varphi_1(\cdot)$ is a positive measurable function. The risk function $R(\delta_{\varphi_1},\sigma_1,\sigma_2)=E^T[E\left\{L\left(V_1^k\varphi_1(T)\right)|T\right\}]$. The conditional risk function can be written as $R_1(d;\sigma_1,\sigma_2)=E_\eta\left\{L(dV_1^k)|T=t\right\}$, where $V_1|T=t \sim Gamma(p_1+p_2-2,(1+\eta t)^{-1})$ where $\eta=\sigma_1/\sigma_2\leq1$. The function $R_1(d,\sigma_1,\sigma_2)$ minimized at $d_{\eta}(t)$, where $d_{\eta}(t)$ be the unique solution of the equation $E_\eta(L^\prime(V_1^kd_\eta(t))V_1^k|T=t)=0$. Using the Lemma of \cite{bobotas2010estimation}, we have
	\begin{align*}
		E_{\eta}\left(L^\prime\left(V_1^kd_1(t)\right)V_1^k|T=t\right)\geq 	E_1\left(L^\prime\left(V_1^kd_1(t)\right)V_1^k|T=t\right)=0=	E_{\eta}\left(L^\prime\left(V_1^kd_\eta(t)\right)V_1^k|T=t\right).
	\end{align*}
	Consequently we get $d_{\eta}(t)\leq d_1(t)$, where $d_1(t)$ is the unique solution of 
	\begin{align*}
		E_1\left(L^\prime\left(V_1^kd_1(t)\right)V_1^k|T=t\right)=0.
	\end{align*}
	Taking the transformation $y_1=v_1(1+t)$, we get $	E_{\eta}\left(L^\prime\left(Y_1^kd_1(t)(1+t)^{-k}\right)\right)=0$ where $Y_1\sim Gamma(p_1+p_2+k-2)$. Comparing with (\ref{st1}), we obtain $d_1(t)=\alpha_1(1+t)^k$. Consider $\varphi_{01}(t)=\min \left\{ \varphi_1(t), d_1(t)\right\}$, then for $P(d_1(T)<\varphi_1(T))\neq0$ we get $d_{\eta}(t)\leq d_1(t)= \varphi_{01}(t)<\varphi_1(t)$ on a set of positive probability. So, we get $R_1(\varphi_{01},\sigma_1,\sigma_2)<R_1(\varphi_1,\sigma_1,\sigma_2)$. Hence we get the results as in the following theorem.	
	\begin{theorem}\label{th1}
		Let $\alpha_1$ be the unique solution of the equation
		\begin{equation}\label{st1}
			EL^\prime\left(Y_1^k\alpha_1\right)=0
		\end{equation}
		where $Y_1\sim Gamma(p_1+p_2+k-2,1)$. Then the risk of the estimator $\delta_{\varphi_{01}}=\varphi_{01}(T)S_1^k$ is nowhere larger than that of the estimator $\delta_{\varphi_1}$ provided $P(\varphi_1(T)>d_1(T))\neq0$ is true.
	\end{theorem}
	\begin{corollary}
		The risk of the estimator $\delta_{11}=\min\left\{d_{01},\alpha_1(1+T)^k\right\}S_1^k$ dominates the BAEE $\delta_{01}$ provided $\alpha_1<d_{01}$.
	\end{corollary}	
	Similar to \cite{petropoulos2017estimation}, we consider a wider class of estimator of the following form 
	\begin{align*}
		\mathcal{D}_2=\left\{ \delta_{\varphi_2}:\ \  \delta_{\varphi_2}\left(X,S\right)=\varphi_2(T,T_1)S_1^k\right\}
	\end{align*}
	where $T_1=\frac{X_1}{S_1}$ and $\varphi_2(\cdot)$ is a positive measurable function.
	\begin{theorem}\label{th}
		Let $Y_2\sim Gamma(p_1+p_2+k-1,1)$ and $\alpha_2$ be a solution of the equation
		\begin{equation}\label{phi2}
			EL^\prime\left(Y_2^k\alpha_2\right)=0
		\end{equation} then the risk of the estimator
		\begin{align*}
			\delta_{\varphi_{02}} = \begin{cases}
				\min\{ \varphi_2(T,T_1), d_{1,0}(T,T_1)\} S_1^k,& T_1>0\\
				\varphi_2(T,T_1)S_1^k,& \mbox{otherwise}
			\end{cases}
		\end{align*}
		is nowhere larger than that of the estimator $\delta_{\varphi_2}$ provided $P(\varphi_2(T,T_1)>d_{1,0}(T,T_1))\neq0$ where $d_{1,0}(T,T_1)=\alpha_2(1+T+p_1T_1)^k$.
	\end{theorem}
	
	\noindent \textbf{Proof:} The risk function of the estimator $\delta_{\varphi_2}$ is
	\begin{eqnarray*}
		R\left(\delta_{\varphi_2},\mu_1,\sigma_1,\sigma_2\right)
		=E\left[E\left\{L\left(V_1^k\varphi_2(T,T_1)\right)\big\rvert T, T_1\right\} \right].
	\end{eqnarray*}
	The conditional risk can be written as  $R_1(d,\mu_1,\sigma_1,\sigma_2)=E\left\{L\left(dV_1^k\right)\big\rvert T=t,T_1=t_1\right\}$ where the conditional density of $V_1$ given $T=t$, $T_1=t_1$, is 
	\begin{align*}
		f_{\eta,\eta_1}(v_1) \propto v_1^{p_1+p_2-2} e^{-v_1(1+\eta t+p_1t_1)},~ t>0,~ v_1>\max\left\{0,\frac{\eta_1}{t_1}\right\}
	\end{align*} 
	where $\eta=\frac{\sigma_1}{\sigma_2}\leq1$, $\eta_1=\frac{\mu_1}{\sigma_1}$.
		Applying Lemma A.2. from \cite{bobotas2010estimation} repeatedly, we get  for all $d>0$
		\begin{align*}
			E_{\eta,\eta_1}\left[L'\left(dV_1^k\right)V_1^k\right] \geq E_{\eta,0}\left[L'\left(dV_1^k\right)V_1^k\right] &\geq E_{1,0}\left[L'\left(dV_1^k\right)V_1^k\right]
		\end{align*}
		Let $d_{\eta,\eta_1}(t,t_1)$ is the unique minimizer of $R_1(d,\mu_1,\sigma_1,\sigma_2)$. Now take $d=d_{1,0}(t,t_1)$ we have
		\begin{align*}
			E_{\eta,\eta_1}\left[L'\left(d_{1,0}(t,t_1)V_1^k\right)V_1^k\right]  &\geq	E_{1,0}\left[L'\left(d_{1,0}(t,t_1)V_1^k\right)V_1^k\right] = 0 = E_{\eta,\eta_1}\left[L'\left(d_{\eta,\eta_1}(t,t_1)V_1^k\right)V_1^k\right].
		\end{align*}
		Since $L^{\prime}(t)$ is strictly increasing then from the above inequality, we have $d_{\eta,\eta_1}(t,t_1)\leq d_{1,0}(t,t_1)$, where $ d_{1,0}(t,t_1)$ is the unique solution of $$E_{1,0}\left[L'\left(d_{1,0}(t,t_1)V_1^k\right)V_1^k\right]=0.$$  
		Using the transformation $y_2=v_1\left(1+t+p_1t_1\right)$ we obtain 
		\begin{equation}
			EL'\left(d_{1,0}(t,t_1)Y_2^k\left(1+t+p_1t_1\right)^{-k}\right)=0,
		\end{equation} 
		where $Y_2\sim \mbox{Gamma}(p_1+p_2+k-1)$. Comparing with equation (\ref{phi2}) we get  
		$d_{1,0}(t,t_1)=\alpha_2\left(1+t+p_1t_1\right)^k.$ Define a function  $\varphi_{02}(t,t_1)=\min\left\{\varphi_2(t,t_1),d_{1,0}(t,t_1)\right\}$. Now we have $$d_{\eta,\eta_1}(t,t_1)\le  d_{1,0}(t,t_1)=\varphi_{02} < \varphi_2(t,t_1)$$  provided $P(d_{1,0}(T,T_1)$ $<\varphi_2(T,T_1))>0$. Hence we get $R_1(\varphi_2,\mu_1,\sigma_1,\sigma_2) > R_1(\varphi_{02},\mu_1,\sigma_1,\sigma_2)$. This completes the proof of the result.
		\begin{corollary}
			The estimator 
			\begin{align*}
				\delta_{12}= \begin{cases}\min \left\{d_{01}, \alpha_2\left(1+T+p_1T_1\right)^k\right\} S_1^k, & T_1>0 \\ 
					d_{01} S_1^k, & \text { otherwise }\end{cases}
			\end{align*}
			is nowhere larger than that of the estimator $\delta_{\varphi_2}$ provided $\alpha_2< d_{01}$.
		\end{corollary}
		If we consider a class of estimator as 
		\begin{align*}
			\mathcal{D}_3=\left\{ \delta_{\varphi_3}:\ \  \delta_{\varphi_3}\left(X,S\right)=\varphi_3(T,T_2)S_1^k; \ T_2=\frac{X_2}{S_1},~\varphi_3(\cdot) \mbox{ is a positive measurable function}\right\}
		\end{align*}
		\begin{corollary}
			The estimator 
			\begin{align*}
				\delta_{13}= \begin{cases}\min \left\{d_{01}, \alpha_3\left(1+T+p_2T_2\right)^k\right\} S_2^k, & T_2>0 \\ 
					d_{01} S_1^k, & \text { otherwise }\end{cases}
			\end{align*}
			is nowhere larger than that of the estimator $\delta_{\varphi_3}$ provided $\alpha_2< d_{01}$.
		\end{corollary}
		Now using the information present in both samples, we consider a larger class of estimators as follows
		\begin{align*}
			\mathcal{D}_4= \left\{\delta_{\varphi_4}:\delta_{\varphi_4}=\varphi_4(T,T_1,T_2)S_1^K, ~\varphi_4(\cdot) \mbox{ is a positive measurable function}\right\}
		\end{align*}
		In the following theorem, we give sufficient conditions under which we will get an improved estimator. 
		\begin{theorem}\label{th2.13}
			Let $Y_4 \sim Gamma(p_1+p_2+k,1)$ and $\alpha_4$ be a solution of the equation 
			\begin{equation}
				EL^{\prime} \left(Y_4^k\alpha_4\right)=0.
			\end{equation}
			Then the risk of the estimator
			$$\delta_{\varphi_{04}}= \begin{cases}\min \left\{\varphi_4(T,T_1,T_2), d_{1,0,0}(T,T_1,T_2)\right\} S_1^k, & T_1>0,T_2>0 \\ \varphi_4(T,T_1,T_2) S_1^k, & \text { otherwise }\end{cases}$$
			is nowhere larger than that of the estimator $\delta_{\varphi_4}$ under the loss function $L(t)$ provided $P(d_{1,0,0}(T,T_1,T_2)<\varphi_4(T,T_1,T_2))\neq0$, where $d_{1,0,0}(T,T_1,T_2)=\alpha_4\left(1+T+p_1T_1+p_2T_2\right)^k$. 	
		\end{theorem}
		\noindent\textbf{Proof.} The proof can be carried out using the same arguments as in Theorem \ref{th}.
		\begin{corollary}
			The estimator 
			\begin{align*}
				\delta_{14}= \begin{cases}\min \left\{d_{01}, \alpha_4\left(1+T+p_1T_1+p_2T_2\right)^k\right\} S_1^k, & T_1>0,\ T_2>0 \\ 
					d_{01} S_1^k, & \text { otherwise }\end{cases}
			\end{align*}
			dominates the BAEE under a general scale invariant loss function $L(t)$ provided $\alpha_3<c_{01}$.
		\end{corollary}
		\begin{example}
			\begin{enumerate}
				\item [(i)] For quadratic loss function $L_1(t)$, we have $\alpha_1= \frac{\Gamma(p_1+p_2+k-2)}{\Gamma(p_1+p_2+2k-2)}$, $\alpha_2= \frac{\Gamma(p_1+p_2+k-1)}{\Gamma(p_1+p_2+2k-1)}$ and $\alpha_4= \frac{\Gamma(p_1+p_2+k)}{\Gamma(p_1+p_2+2k)}$. The improved estimators are obtained as follows
				\begin{align*}
					&\delta_{11}^1(X,S)=\min\left\{\frac{\Gamma(p_1+k-1)}{\Gamma(p_1+2k-1)}, \alpha_1(1+T)^k\right\}S_1^k\\
					&\delta_{12}^1(X,S)=\begin{cases}\min \left\{\frac{\Gamma(p_1+k-1)}{\Gamma(p_1+2k-1)}, \alpha_2\left(1+T+p_1T_1\right)^k\right\} S_1^k, & T_1>0 \\ 
						\frac{\Gamma(p_1+k-1)}{\Gamma(p_1+2k-1)} S_1^k, & \text { otherwise }\end{cases}\\
					&\delta_{13}^1(X,S)=\begin{cases}\min \left\{\frac{\Gamma(p_1+k-1)}{\Gamma(p_1+2k-1)}, \alpha_2\left(1+T+p_2T_2\right)^k\right\} S_1^k, & T_2>0 \\ 
						\frac{\Gamma(p_1+k-1)}{\Gamma(p_1+2k-1)} S_1^k, & \text { otherwise }\end{cases}\\
					&\delta_{14}^1(X,S)=\begin{cases}\min \left\{\frac{\Gamma(p_1+k-1)}{\Gamma(p_1+2k-1)}, \alpha_4\left(1+T+p_1T_1+p_2T_2\right)^k\right\} S_1^k, & T_1>0,\ T_2>0 \\ 
						\frac{\Gamma(p_1+k-1)}{\Gamma(p_1+2k-1)} S_1^k, & \text { otherwise }\end{cases}\\
				\end{align*}
				\item [(ii)] For entropy loss function $L_2(t)$, we have $\alpha_1= \frac{\Gamma(p_1+p_2-2)}{\Gamma(p_1+p_2+k-2)}$, $\alpha_2= \frac{\Gamma(p_1+p_2-1)}{\Gamma(p_1+p_2+k-1)}$ and $\alpha_4= \frac{\Gamma(p_1+p_2)}{\Gamma(p_1+p_2+k)}$. The improved estimators are obtained as follows
				\begin{align*}
					&\delta_{11}^2(X,S)=\min\left\{\frac{\Gamma(p_1-1)}{\Gamma(p_1+k-1)}, \alpha_1(1+T)^k\right\}S_1^k\\
					&\delta_{12}^2(X,S)=\begin{cases}\min \left\{\frac{\Gamma(p_1-1)}{\Gamma(p_1+k-1)}, \alpha_2\left(1+T+p_1T_1\right)^k\right\} S_1^k, & T_1>0 \\ 
						\frac{\Gamma(p_1-1)}{\Gamma(p_1+k-1)} S_1^k, & \text { otherwise }\end{cases}\\
					&\delta_{13}^2(X,S)=\begin{cases}\min \left\{\frac{\Gamma(p_1-1)}{\Gamma(p_1+k-1)}, \alpha_2\left(1+T+p_2T_2\right)^k\right\} S_1^k, & T_2>0 \\ 
						\frac{\Gamma(p_1-1)}{\Gamma(p_1+k-1)} S_1^k, & \text { otherwise }\end{cases}\\
					&\delta_{14}^2(X,S)=\begin{cases}\min \left\{\frac{\Gamma(p_1-1)}{\Gamma(p_1+k-1)}, \alpha_4\left(1+T+p_1T_1+p_2T_2\right)^k\right\} S_1^k, & T_1>0,\ T_2>0 \\ 
						\frac{\Gamma(p_1-1)}{\Gamma(p_1+k-1)} S_1^k, & \text { otherwise }\end{cases}\\
				\end{align*}
				\item [(iii)] For symmetric loss function $L_3(t)$, we have $\alpha_1= \sqrt{\frac{\Gamma(p_1+p_2-k-2)}{\Gamma(p_1+p_2+k-2)}}$, $\alpha_2= \sqrt{\frac{\Gamma(p_1+p_2-k-1)}{\Gamma(p_1+p_2+k-1)}}$ and $\alpha_4=\sqrt{\frac{\Gamma(p_1+p_2-k)}{\Gamma(p_1+p_2+k)}}$. The improved estimators are obtained as follows
				\begin{align*}
					&\delta_{11}^3(X,S)=\min\left\{\sqrt{\frac{\Gamma(p_1-k-1)}{\Gamma(p_1+k-1)}} , \alpha_1(1+T)^k\right\}S_1^k\\
					&\delta_{12}^3(X,S)=\begin{cases}\min \left\{\sqrt{\frac{\Gamma(p_1-k-1)}{\Gamma(p_1+k-1)}} , \alpha_2\left(1+T+p_1T_1\right)^k\right\} S_1^k, & T_1>0 \\ 
						\sqrt{\frac{\Gamma(p_1-k-1)}{\Gamma(p_1+k-1)}} S_1^k, & \text { otherwise }\end{cases}\\
					&\delta_{13}^3(X,S)=\begin{cases}\min \left\{\sqrt{\frac{\Gamma(p_1-k-1)}{\Gamma(p_1+k-1)}} , \alpha_2\left(1+T+p_2T_2\right)^k\right\} S_1^k, & T_2>0 \\ 
						\sqrt{\frac{\Gamma(p_1-k-1)}{\Gamma(p_1+k-1)}} S_1^k, & \text { otherwise }\end{cases}\\
					&\delta_{14}^3(X,S)=\begin{cases}\min \left\{\sqrt{\frac{\Gamma(p_1-k-1)}{\Gamma(p_1+k-1)}} , \alpha_4\left(1+T+p_1T_1+p_2T_2\right)^k\right\} S_1^k, & T_1>0,\ T_2>0 \\ 
						\sqrt{\frac{\Gamma(p_1-k-1)}{\Gamma(p_1+k-1)}} S_1^k, & \text { otherwise }\end{cases}\\
				\end{align*}
				\item [(iv)] For linex loss function $L_4(t)$, the improved estimators are obtained as follows
				\begin{align*}
					&\delta_{11}^4(X,S)=\min\left\{ d_{01}, \alpha_1(1+T)^k\right\}S_1^k\\
					&\delta_{12}^4(X,S)=\begin{cases}\min \left\{d_{01} , \alpha_2\left(1+T+p_1T_1\right)^k\right\} S_1^k, & T_1>0 \\ 
						d_{01}S_1^k, & \text { otherwise }\end{cases}\\
					&\delta_{13}^4(X,S)=\begin{cases}\min \left\{d_{01} , \alpha_2\left(1+T+p_2T_2\right)^k\right\} S_1^k, & T_2>0 \\ 
						d_{01}S_1^k, & \text { otherwise }\end{cases}\\
					&\delta_{14}^4(X,S)=\begin{cases}\min \left\{d_{01} , \alpha_4\left(1+T+p_1T_1+p_2T_2\right)^k\right\} S_1^k, & T_1>0,\ T_2>0 \\ 
						d_{01} S_1^k, & \text { otherwise }\end{cases}\\
				\end{align*}
				where $\alpha_1$, $\alpha_2$ and $\alpha_4$ are a solution of the following equation respectively
				\begin{align*}
					&\int_{0}^{\infty} y_1^{p_1+p_2+k-3} e^{-y_1+\alpha(\alpha_1y_1^k-1)}dy_1=1\\
					&\int_{0}^{\infty} y_2^{p_1+p_2+k-2} e^{-y_2+\alpha(\alpha_2y_2^k-1)}dy_2=1\\
					&\int_{0}^{\infty} y_4^{p_1+p_2+k-2} e^{-y_4+\alpha(\alpha_4y_4^k-1)}dy_4=1\\
				\end{align*}
			\end{enumerate}
		\end{example}
		In the following theorem, we derive a class of improved estimators using the IERD approach \cite{kubokawa1994unified}. The joint density of $V_1$ and $T$ is 
		\begin{align}
			f_{\eta}(t,v_1)\propto \eta^{p_2-1} e^{-v_1(1+\eta t)}t^{p_2-2}v_1^{p_1+p_2-3},\ \ v_1>0,\ \ t>0,\ \ 0<\eta\leq1
		\end{align}
		Define,
		\begin{align*}
			F_{\eta}(x,v_1)=\int_{0}^{x}f_{\eta}(u,v_1)du
		\end{align*}
		\begin{theorem}\label{th3.3}
			Let the function $\varphi_1$ satisfies the following conditions.
			\begin{enumerate}
				
				\item[(i)] $\varphi_1(t)$ is increasing function in $t$ and $\lim\limits_{t\rightarrow \infty}\varphi_1(t)=d_{01}$.
				
				\item[(ii)] $\int_{0}^{\infty}L'(\varphi_1(t)v_1^k)v_1^k F_1(t,v_1)dv_1\geq0$.
			\end{enumerate}
			Then the risk of $\delta_{\varphi_1}$ in $\mathcal{D}_1$ is uniformly smaller than the estimator $\delta_{01}$ under $L(t)$.
		\end{theorem}
		\noindent\textbf{Proof.} The proof proceeds in a similar manner to that of Theorem 4.1 in \cite{kubokawa1994double}

		\begin{corollary}
			Assume that the function $\varphi_1(t)$ satisfies the following conditions:
			\begin{enumerate}
				\item[(i)] $\varphi_1(t)$ is non decreasing and $\lim\limits_{t\rightarrow \infty}\varphi_1(t)= \frac{\Gamma(p_1+k-1)}{\Gamma(p_1+2k-1)}$.
				\item [(ii)] $\varphi_1(t)\geq \varphi_{01}^1(t)=\frac{\Gamma(p_1+p_2+k-2) B(p_2-1,p_1+k-1)}{\Gamma(p_1+p_2+2k-2) B(p_2-1,p_1+k-1)}$
				
			\end{enumerate}
			Then the risk of the estimator $\delta_{\varphi_{1}}$ given in $\mathcal{D}_1$ is nowhere larger than that of the BAEE $\delta_{01}^1$.
		\end{corollary}
		
		\begin{corollary}
			Assume that the function $\varphi_1(t)$ satisfies the following conditions:
			\begin{enumerate}
				\item[(i)] $\varphi_1(t)$ is non decreasing and $\lim\limits_{t\rightarrow \infty}\varphi_1(t)= \frac{\Gamma(p_1-1)}{\Gamma(p_1+k-1)}$
				\item [(ii)] $\varphi_1(t)\geq \varphi_{01}^2(t)= \frac{\Gamma(p_1+p_2-2) B(p_2-1,p_1-1)}{\Gamma(p_1+p_2+k-2) B(p_2-1,p_1+k-1)}$
				
			\end{enumerate}
			Then under the loss function $L_2(t)$, the risk of the estimator $\delta_{\varphi_{1}}$ given in $\mathcal{D}_1$ is nowhere larger than that of the BAEE $\delta_{01}^2$.
		\end{corollary}
		
		\begin{corollary}
			Under the loss function $L_3(t)$, the risk of the estimator $\delta_{\varphi_1}$ given in $\mathcal{D}_1$ is nowhere larger than that of $\delta_{01}^3$, provide the function $\varphi_1(t)$ satisfies the following properties
			\begin{enumerate}
				\item[(i)] $\varphi_1(t)$ is non decreasing and $\lim\limits_{t\rightarrow \infty}\varphi_1(t) = \sqrt{\frac{\Gamma(p_1-k-1)}{\Gamma(p_1+k-1)}}$
				\item [(ii)] $\varphi_1(t) \geq \varphi_{01}^3(t) = \sqrt{\frac{\Gamma(p_1+p_2-k-2) B(p_2-1,p_1-k-1)}{\Gamma(p_1+p_2+k-2) B(p_2-1,p_1+k-1)}}$
			\end{enumerate}
			Then the risk of the estimator $\delta_{\varphi_{1}}$ given in $\mathcal{D}_1$ is nowhere larger than that of the BAEE $\delta_{01}^3$.
		\end{corollary}
		
		\begin{remark}
			For the linex loss function $L_4(t)$ the closed form are not be to find.
		\end{remark}
		
		\begin{remark}\label{bz1}
			From the above corollaries, we derive a class of improved estimators under the loss functions $L_1(t)$, $L_2(t)$ and $L_3(t)$. The corresponding boundary estimators are $\delta_{\varphi_{01}^1}=\varphi_{01}^1S_1^k$, $\delta_{\varphi_{01}^2}=\varphi_{01}^2S_1^k$ and $\delta_{\varphi_{01}^3}=\varphi_{01}^3S_1^k$. These boundary estimators are of the \cite{brewster1974improving}-type.
		\end{remark}	
		\begin{remark}
			In Remark \ref{bz1}, the \cite{brewster1974improving}-type estimator $\delta_{\varphi_{01}}$ is a generalized Bayes estimator for $\sigma_1^k$ under the non-informative prior 
			\begin{align*}
				\Pi(\mu_1,\mu_2,\sigma_1,\sigma_2)=\frac{1}{\sigma_1\sigma_2},~~ \mu_1<X_1,~ \mu_2<X_2,~ \text{and}~ 0<\sigma_1\leq\sigma_2.
			\end{align*}
		\end{remark}
		
		\section{Improved estimation of $\sigma_2^k$}\label{se3}
		We consider the class of estimator of the form 
		\begin{align*}
			\mathcal{C}_1=\left\{ \delta_{\xi_1}:\ \  \delta_{\xi_1}\left(X,S\right)=\xi_1(W)S_2^k\right\}
		\end{align*}
		where $W=\frac{S_1}{S_2}$ and $\xi_1(\cdot)$ is a positive measurable function.
		
		\begin{theorem}\label{th2}
			Let $Z_1$ be a random variable having gamma distribution $Gamma(p_1+p_2+k-2,1)$ and $\beta_{1}$ be the unique solution of 
			\begin{equation}
				E(L'(\beta_{1}Z_1^k))=0
			\end{equation}
			Consider $\xi_{01}(W)=\max\left\{ \xi_1(W), d_2(W)\right\}$, then the risk of the estimator $\delta_{\xi_{01}}=\xi_{01}(W)S_2^k$ is nowhere larger than that of the estimator $\delta_{\xi_1}$ provided $P(\xi_1(W)<d_1(W))\neq0$ holds true.
		\end{theorem}
		\noindent \textbf{Proof.} The proof can be carried out using the same arguments as in Theorem \ref{th1}.
		
		\begin{corollary}\label{coro3.2}
			The risk function of the estimator
			$\delta_{21}=\max \left\{d_{02}, \beta_1(1+W)^k\right\} S_2^k$ dominates estimator $\delta_{02}$ provided $\beta_1 < d_{02}$.
		\end{corollary}
		In the Theorem \ref{th2}, we have obtained an improved estimator $\sigma_2^k$ using the information contained only in $S_1$ and $S_2$. Now we use the information contained in $X_2$, for that we consider the class of estimator of the form 
		\begin{align*}
			\mathcal{C}_2=\left\{ \delta_{\xi_2}:\ \  \delta_{\xi_2}\left(X,S\right)=\xi_2(W_1)S_2^k\right\}
		\end{align*}
		where $W_1=\frac{X_2}{S_2}$ and $\xi_2(\cdot)$ is a measurable function.
		\begin{theorem}\label{th3}
			Let $Z_2$ be a random variable having gamma distribution $Gamma(p_2+k,1)$ and $\beta_{2}$ be the unique solution of 
			\begin{equation}
				E(L'(\beta_2Z_2^k))=0
			\end{equation}
			Consider $\xi_{02}(W_1)=\min\left\{ \xi_2(W_1), d_2(W_1)\right\}$, then the risk of the estimator $\delta_{\xi_{01}}=\xi_{01}(W_1)S_2^k$ is nowhere larger than that of the estimator $\delta_{\xi_2}$ provided $P(\xi_2(W_1)<d_2(W_1))\neq0$ holds true.
		\end{theorem}
		\noindent\textbf{Proof.} The proof can be carried out using the same arguments as in Theorem \ref{th}.
		\begin{corollary}\label{coro3.3}
			The risk function of the estimator
			$\delta_{22}=\min \left\{d_{02}, \beta_2(1+W_1)^k\right\} S_2^k$ dominates estimator $\delta_{02}$ provided $\beta_2 < d_{02}$.
		\end{corollary}
		\begin{example}
			\begin{enumerate}
				\item[(i)] For quadratic loss function $L_1(t)$, we have $\beta_1= \frac{\Gamma(p_1+p_2+k-2)}{\Gamma(p_1+p_2+2k-2)}$, $\beta_2= \frac{\Gamma(p_2+k)}{\Gamma(p_2+2k)}$. The improved estimators are obtained as follows
				\begin{align*}
					&\delta_{21}^1(X,S)=\max\left\{\frac{\Gamma(p_2+k-1)}{\Gamma(p_2+2k-1)}, \beta_1(1+W)^k\right\}S_2^k\\
					&\delta_{22}^1(X,S)=\begin{cases}\min \left\{\frac{\Gamma(p_2+k-1)}{\Gamma(p_2+2k-1)}, \beta_2\left(1+W_1\right)^k\right\} S_2^k, & W_1>0 \\ 
						\frac{\Gamma(p_2+k-1)}{\Gamma(p_2+2k-1)} S_2^k, & \text { otherwise }\end{cases}\\
				\end{align*}
				\item[(ii)] For entropy loss function $L_2(t)$, we have $\beta_1= \frac{\Gamma(p_1+p_2-2)}{\Gamma(p_1+p_2+k-2)}$, $\beta_2= \frac{\Gamma(p_2)}{\Gamma(p_2+k)}$. The improved estimators are obtained as follows
				\begin{align*}
					&\delta_{21}^2(X,S)=\max\left\{\frac{\Gamma(p_2-1)}{\Gamma(p_2+k-1)}, \beta_1(1+W)^k\right\}S_2^k\\
					&\delta_{22}^2(X,S)=\begin{cases}\min \left\{\frac{\Gamma(p_2-1)}{\Gamma(p_2+k-1)}, \beta_2\left(1+W_1\right)^k\right\} S_2^k, & W_1>0 \\ 
						\frac{\Gamma(p_2-1)}{\Gamma(p_2+k-1)} S_2^k, & \text { otherwise }\end{cases}\\
				\end{align*}
				\item[(iii)] For symmetric loss function $L_3(t)$, we have $\beta_1= \sqrt{\frac{\Gamma(p_1+p_2-k-2)}{\Gamma(p_1+p_2+k-2)}}$, $\beta_2= \sqrt{\frac{\Gamma(p_2-k)}{\Gamma(p_2+k)}}$. The improved estimators are obtained as follows
				\begin{align*}
					&\delta_{21}^3(X,S)=\max\left\{\sqrt{\frac{\Gamma(p_2-k-1)}{\Gamma(p_2+k-1)}}, \beta_1(1+W)^k\right\}S_2^k\\
					&\delta_{22}^3(X,S)=\begin{cases}\min \left\{\sqrt{\frac{\Gamma(p_2-k-1)}{\Gamma(p_2+k-1)}}, \beta_2\left(1+W_1\right)^k\right\} S_2^k, & W_1>0 \\ 
						\sqrt{\frac{\Gamma(p_2-k-1)}{\Gamma(p_2+k-1)}} S_2^k, & \text { otherwise }\end{cases}\\
				\end{align*}
				\item[(iv)] For linex loss function $L_4(t)$, the improved estimators are obtained as follows
				\begin{align*}
					&\delta_{21}^4(X,S)=\min\left\{d_{02}, \beta_1(1+W)^k\right\}S_2^k\\
					&\delta_{22}^4(X,S)=\begin{cases}\min \left\{d_{02}, \beta_2\left(1+W_1\right)^k\right\} S_2^k, & W_1>0 \\ 
						d_{02} S_2^k, & \text { otherwise }\end{cases}\\
				\end{align*}
				where $\beta_1$ and $\beta_2$ are solution of the following equations respectively
				\begin{align*}
					&\int_{0}^{\infty} z_1^{p_1+p_2+k-3} e^{-z_1+\alpha(\beta_1z_1^k-1)}dz_1=1\\
					&\int_{0}^{\infty} z_2^{p_2+k-1} e^{-z_2+\alpha(\beta_2z_2^k-1)}dz_2=1\\
				\end{align*}
			\end{enumerate}
		\end{example}
		In the following theorem, we derive a class of improved estimators using the IERD approach \cite{kubokawa1994unified}.
		\begin{theorem}\label{t3.3}
			Let the function $\xi_1$ satisfies the following conditions.
			\begin{enumerate}
				\item[(i)] $\xi_1(w)$ is increasing function in $w$ and $\lim\limits_{w\rightarrow 0}\xi_1(w)=d_{02}$.
				\item[(ii)] $\int_{0}^{\infty}\int_{v_2w}^{\infty}L'(\xi_1(w)v_2^k)v_2^{k}\nu_1(y)\nu_2(v_2)dydv_2\leq0$.
			\end{enumerate}
			where $\nu_i$ is pdf of a $\mbox{Gamma}(p_i-1)$ distribution for $i=1,2$. Then the risk of $\delta_{\xi_1}$ in $\mathcal{C}_1$ is uniformly smaller than the estimator $\delta_{02}$ under $L(t)$.
		\end{theorem}
		\noindent \textbf{Proof:} Proof of this theorem is similar to the Theorem 4.3 of \cite{kubokawa1994double}	
		In the following, we have obtained  improved estimators for $\sigma_2^k$ under three special loss functions by applying Theorem \ref{t3.3}. 
		\begin{corollary}
			Let us assume that the function $\xi_1(w)$  satisfies the subsequent criterion:
			\begin{enumerate}
				\item[(i)]$\xi_1(w)$ is increasing function in $w$ and $\lim\limits_{w\rightarrow0}\xi_1(w)=\frac{\Gamma\left(p_2+k-1\right)}{\Gamma\left(p_2+2k-1\right)} $
				\item[(ii)]$\xi_1(w)\leq\xi_{01}^1(w)$
			\end{enumerate}
			where
			\begin{align*}
				\xi_{01}^1(w)=\frac{\int_{0}^{\infty}\int_{v_2w}^{\infty}x^{p_1-2}e^{-x}v_2^{p_2+k-2}e^{-v_2}dxdv_2}{\int_{0}^{\infty}\int_{v_2w}^{\infty}x^{p_1-2}e^{-x}v_2^{p_2+2k-2}e^{-v_2}dxdv_2}
			\end{align*}
			Then under the loss function $L_1(t)$, the risk of the estimator  $\delta_{\xi_1}$ is nowhere larger than that of $\delta^1_{02}$.
		\end{corollary}
		\begin{corollary}
			Let us assume that the function $\xi_1(w)$ satisfies the following conditions 
			\begin{enumerate}
				\item[(i)] $\xi_1(w)$ is increasing function in $w$ and $\lim\limits_{w\rightarrow0}\xi_1(w)=\frac{\Gamma\left(p_2-1\right)}{\Gamma\left(p_2+k-1\right)}$
				\item[(ii)] $\xi_1(w)\leq\xi_{01}^2(w)$
			\end{enumerate}
			where
			\begin{align*}
				\xi_{01}^2(w)=\frac{\int_{0}^{\infty}\int_{v_2w}^{\infty}x^{p_1-2}e^{-x}v_2^{p_2-2}e^{-v_2}dxdv_2}{\int_{0}^{\infty}\int_{v_2w}^{\infty}x^{p_1-2}e^{-x}v_2^{p_2+k-2}e^{-v_2}dxdv_2}
			\end{align*}
			The risk of the estimator  $\delta_{\xi_1}$ is uniformly smaller than that of $\delta^2_{02}$ with respect to $L_2(t)$.
		\end{corollary}
		\begin{corollary}
			Let us assume that the following conditions holds true
			\begin{enumerate}
				\item[(i)] $\xi_1(w)$ is increasing function in $w$ and $\lim\limits_{w\rightarrow0}\xi_1(w)=\sqrt{\frac{\Gamma\left(p_2-k-1\right)}{\Gamma\left(p_2+k-1\right)}}$.
				\item[(ii)] $\xi_1(w)\leq\xi_{01}^3(w)$
			\end{enumerate}
			where
			\begin{align*}
				\xi_{01}^3(w)=\sqrt{\frac{\int_{0}^{\infty}\int_{v_2w}^{\infty}x^{p_1-2}e^{-x}v_2^{p_2-k-2}e^{-v_2}dxdv_2}{\int_{0}^{\infty}\int_{v_2w}^{\infty}x^{p_1-2}e^{-x}v_2^{p_2+k-2}e^{-v_2}dxdv_2}}
			\end{align*}	
			The risk of the estimator  $\delta_{\xi_1}$ is nowhere larger  than that of $\delta^3_{02}$ with respect to the loss function $L_3(t)$.
		\end{corollary}
		\begin{remark}
			For the linex loss function $L_4(t)$ the closed form are not be to find.
		\end{remark}
		\begin{remark}\label{bz2}
			From the above corollaries, we derive a class of improved estimators under the loss functions $L_1(t)$, $L_2(t)$ and $L_3(t)$. The corresponding boundary estimators are $\delta_{\xi_{01}^1}=\xi_{01}^1S_2^k$, $\delta_{\xi_{01}^2}=\xi_{01}^2S_2^k$ and $\delta_{\xi_{01}^3}=\xi_{01}^3S_2^k$. These boundary estimators are of the \cite{brewster1974improving}-type.
		\end{remark}	
		\begin{remark}
			In Remark \ref{bz2}, the \cite{brewster1974improving}-type estimator $\delta_{\xi_{01}}$ is a generalized Bayes estimator for $\sigma_2^k$ under the non-informative prior 
			\begin{align*}
				\Pi(\mu_1,\mu_2,\sigma_1,\sigma_2)=\frac{1}{\sigma_1\sigma_2},~~ \mu_1<X_1,~ \mu_2<X_2,~ \text{and}~ 0<\sigma_1\leq\sigma_2.
			\end{align*}
		\end{remark}
		\subsection{Double shrinkage improved estimators}
		In this section we will propose a double shrinkage estimator similar to \cite{iliopoulos1999improving}. We will find the double shrinkage estimator by jointly using the aforementioned improved estimators. Consider the following estimators $\delta_{2S1}$, $\delta_{2S2}$ which is derived in the corollary \ref{coro3.2} and \ref{coro3.3} respectively. Let $\xi_1(w)$, $\xi_2(w_1)$ and $L(t)$ are almost every where differentiable along with $\xi_1(w)$ is non-increasing and $\xi_2(w_1)$ is non-decreasing. Also satisfies the conditions $\lim\limits_{w \to \infty}\xi_1(w)=\lim\limits_{w_1\to \infty}\xi_2(w_1)=c_{02}$. Then we have the following results.
		\begin{theorem} The estimators $\delta_{D}=\left\{\xi_1(W)+\xi_2(W_1)-c_{02}\right\}S_2^{\frac{k}{2}}$ dominates both $\delta_{2S1}$ and $\delta_{2S2}$ provided $L'(t)$ is non decreasing.
		\end{theorem}
		\noindent \textbf{Proof:} The theorem can be proof by using the similar argument as used in \cite{iliopoulos1999improving}. 	
		\section{On improved estimation using generalized Pitman closeness}\label{se4}
		In this section we derive the estimation of $\tau$ under the generalized Pitman closeness criterion. A concise discussion of the Pitman closeness criterion can be found in \cite{garg2024unified}. The Pitman closeness criterion was first proposed by \cite{pitman1937closest}. The generalized Pitman closeness (GPC) criterion (see \cite{nayak1990estimation} and \cite{kubokawa1991equivariant}) under a location invariant loss function $W(\cdot,\cdot)$ is defined as follows. 
		\begin{defn}
			Let $X$ be a random variable having a distribution depends on an unknown parameter $\theta \in \Theta$. Let $\delta_1$ and $\delta_2$ be two estimators of a real-valued estimand $\xi(\theta)$. Also, let $G(\delta,\xi(\theta))$ be a location invariant loss function for estimating $\xi(\theta)$. Then, the GPC of $\delta_1$ relative to $\delta_2$ is defined by
			\begin{align*}
				\mbox{GPC}(\delta_1,\delta_2;\theta)= P_{\theta}\big[G(\theta,\delta_1)<G(\theta,\delta_2)\big] + \frac{1}{2}P_{\theta}\big[G(\theta,\delta_1)=G(\theta,\delta_2)\big], \ \ \theta\in\Theta.
			\end{align*}
			The estimator $\delta_1$ is said to be closer to $\xi(\theta)$ than the estimator $\delta_2$, under the GPC criterion, if $\mbox{GPC}(\delta_1,\delta_2;\theta)\geq \frac{1}{2}$ $\forall \  \theta \in \Theta$, and strict inequality hold for some $\theta \in \Theta$.
		\end{defn}
		We need the following lemma taken from \cite{garg2024unified} to prove the main result of this section.
		\begin{lemma}(\cite{garg2024unified})\label{pt1} Let $Y$ be a random variable having a probability density function and let $m_Y$ be the median of $Y$. Let $L(t)$ be a non negative function such that $L(1)=0$, $L(t)$ is strictly increasing for $t>1$ and strictly decreasing for $t<1$. Then, for $-\infty<d_1<d_2\leq m_Y$ or $-\infty<m_Y\leq d_2<d_1$, $\mbox{GPC}=P[L(Y^kd_2)<L(Y^kd_1)]+\frac{1}{2}P[L(Y^kd_2)=L(Y^kd_1)]>\frac{1}{2}$.
		\end{lemma}
		Note that, the problem of estimating $\tau$ for unrestricted case (the parametric space $\Theta=\mathbb{R}\times\mathbb{R}\times\mathbb{R}^+$), the GPC criterion is invariant under the group of transformation $\mathcal{G}$. Any location equivariant estimator is of the form $\delta_d=dS_1^k$. Under the unrestricted parameter space $\Theta$, an immediate consequence of the Lemma \ref{pt1} is that the Pitman closest equivariant estimator (PCE) of the location parameter $\tau$, under the GPC criterion is $\delta_{PCE}=m_0S_1^k$, where $m_0$ is the median of the random variable $V=\frac{S_1}{\sigma_1}$. We denote $\underline{\gamma}=(\mu_1,\mu_2,\sigma)$ and $\Theta_0=\big\{(\mu_1,\mu_2,\sigma):\ \mu_1,\ \mu_2\in \mathbb{R},\ 0<\sigma<\infty\}$.

		\subsection{Improve estimation of $\sigma_1^k$}
		Let $\delta_{\varphi_1}(\underline{X},S)=\varphi_1(T)S_1^k$ and $\delta_{\varphi_2}(\underline{X},S)=\varphi_2(T)S_1^k$ be two scale invariant estimators of $\sigma_1^k$, where $\varphi_1$ and $\varphi_2$ are two positive real-valued functions defined on $\mathbb{R}$. Then, the GPC of $\delta_{\varphi_1}(X,S)$ relative to $\delta_{\varphi_2}(X,S)$ is given by
		\begin{align*}
			\mbox{GPC}(\delta_{\varphi_1},\delta_{\varphi_2};\eta)&= P_{\eta}\left[L(\frac{\varphi_1(T)S_1^k}{\sigma_1^k})  < L(\frac{\varphi_2(T)S_1^k}{\sigma_1^k})\right]\  + \frac{1}{2} P_{\eta}\left[L(\frac{\varphi_1(T)S_1^k}{\sigma_1^k})  = L(\frac{\varphi_2(T)S_1^k}{\sigma_1^k})\right]\\ 
			&=E^T\Big[P\big[L(\varphi_1(T)V_1^k) < L(\varphi_2(T)V_1^k)|T\big]\Big]\ + \frac{1}{2} E^T\Big[P\big[L(\varphi_2(T)V_1^k) < L(\varphi_2(T)V_1^k)|T\big]\Big]\\
		\end{align*}
		For any fixed positive $t$, we define 
		\begin{align*}
			r_{\eta}(\varphi_1(t),\varphi_2(t),t) = P_{\eta}\big[L(\varphi_1(t)V_1^k) &< L(\varphi_2(t)V_1^k)\big| T=t\big]\ + \frac{1}{2} P_{\eta}\big[L(\varphi_1(t) V_1^k) = L(\varphi_2(t)V_1^k)\big | T=t\big]
		\end{align*}
		For any fixed $\eta\leq1$ and $t>0$, let $m_\eta(t)$ denote the median of the distribution $V_1$ given $T=t$. The conditional distribution $V_1$ given $T=t$ is 
		$$f_{\eta}(v_1|t)\propto \eta^{p_2-1} e^{-v_1(1+\eta t)}v_1^{p_1+p_2-3},\ \ v_1>0,$$
		Thus, \begin{align*}
			\int_{0}^{m_\eta(w)}e^{-v_1(1+\eta t)}v_1^{p_1+p_2-3}dv_1 =\frac{1}{2}\int_{0}^{\infty} e^{-v_1(1+\eta t)}v_1^{p_1+p_2-3}dv_1.
		\end{align*}
		For any fixed $t>0$, by using the Lemma \ref{pt1}, we obtain $r_{\eta}\big(\varphi_1(t),\varphi_2(t),t\big) >1/2$ $\forall \ 0<\eta \leq 1$, if $\varphi_2(t)<\varphi_1(t)\leq m_\eta(t)$ $\forall \ \eta \geq 0$ or if $m_\eta(t)\leq\varphi_1(t)<\varphi_2(t)$ $\forall \ \eta \geq 0$. Moreover, for any fixed $t$, $r_{\eta}(\varphi_1(t),\varphi_2(t),t)=1/2$ $\forall \ 0<\eta \leq 1$.
		\begin{theorem}\label{thpt1}
			Let $\delta_{\varphi_1}(X,S)=\varphi_1(T)S_1^k$ be a scale equivariant estimator of $\sigma_1^k$. For any $t$, let $l(t)$ and $u(t)$ be function such that $l(t)\leq m_\eta(t)\leq u(t)$ $\forall \ \eta \leq 1$. At any fixed $t$, we define $\widetilde{\varphi_1}(t)=\max\left\{l(t),\min\left\{\varphi_1(t), u(t)\right\}\right\}$. Then under GPC criterion with a general loss function $L(t)$ the estimator $\delta_{\widetilde{\varphi_1}}(X,S)= \widetilde{\varphi_1}(T)S_1^k$ is Pitman closest to $\sigma_1^k$ than the estimator $\delta_{\varphi_1}(X,S)=\varphi_1(T)S_1^k$ $\forall \ \underline{\gamma} \in \Theta_0$, provided $P\big[l(t)\leq \varphi_1(t)\leq u(t)\big]<1$ $\forall \ \underline{\gamma} \in \Theta_0$.
		\end{theorem}
		\noindent \textit{Proof:} The GPC of the estimator $\delta_{\widetilde{\varphi_1}}(X,S)=\widetilde{\varphi_1}(T)S_1^k$ relative to $\delta_{\varphi_1}(X,S)=\varphi_1(T)S_1^k$ can be written as $\mbox{GPC}(\delta_{\widetilde{\varphi_1}},\delta_{\varphi_1},\theta)=\int_{0}^{\infty}r_{\eta}(\widetilde{\varphi_1}(t),\varphi_1(t),t)f_t(t|\eta) dt,\ \eta \leq 1$. Let $A = \left\{t: \varphi_!(t)<l(t)\right\}$, $B=\left\{t: l(t)<\varphi_1(t)<u(t)\right\}$, $C=\left\{t: \varphi_1(t)>u(t)\right\}$. It is clear to us 
		\begin{equation*}
			\widetilde{\varphi_1}(t) = \begin{cases} 
				l(t), & t \in A,\\
				\varphi_1(t), & t\in B,\\
				u(t), & t\in C.
			\end{cases}
		\end{equation*} 
		Because $l(t)\leq m_{\eta}(t)\leq u(t)$ $\forall \ \eta \leq 1$ and $t$. Using Lemma \ref{pt1}, we have, $r_{\eta}(\widetilde{\varphi_1}(t),\varphi_1(t),t)>\frac{1}{2}\ \forall\ \eta \leq 1$ provided $t \in A \cup C$.
		For $t\in B$, $r_{1,\eta}(\widetilde{\varphi_1}(t),\varphi_1(t),t)=\frac{1}{2}$ $\forall \ \eta \leq 1$. Again since $P_{\underline{\gamma}}(A\cup C)>0$ $\forall\ \underline{\gamma} \in \Theta_0$ we have
		\begin{align*}
			\mbox{GPC}(\delta_{\widetilde{\varphi_1}},\delta_{\varphi_1},\underline{\gamma})&=\int_Ar_{1,\eta}(\widetilde{\varphi_1}(t),\varphi_1(t),t)f_t(t|\eta) dt +\int_B r_{\eta}(\widetilde{\varphi_1}(t),\varphi_1(t),t)f_t(t|\eta) dt \\
			&~~~~~~~~~~~~~~~~~~~~~~~~~~~~~~~~~~~~~~~~~~~~~~~~~~~~~~~~~~+\int_Cr_{\eta}(\widetilde{\varphi_1}(t),\varphi_1(t),t)f_t(t|\eta) dt\\
			&>\frac{1}{2} ~~~~~~\forall~ \underline{\gamma}\in \Theta_0
		\end{align*} 
		\begin{corollary}
			Let $l(t)$ and $u(t)$ be as defined as above, suppose that $P\big[l(t)\leq m_{01} \leq u(t)\big]<1$ $\forall \ \underline{\gamma} \in \Theta_0$. For any fixed $t>0$ defined $\widetilde{\varphi_1}(t)=\max \left\{ l(t), \min\left\{ m_{01},u(t)\right\}\right\}=\min\left\{m_{01},u(t)\right\}$. Then for any every $\underline{\gamma} \in \Theta_0$ the estimator $\delta_{\widetilde{\varphi_1}}(X,S)=\widetilde{\varphi_1}(T)S_1^k$ is Pitman closest to $\sigma_1^k$ than the (Pitman closest afine equivariant estimator) PCAEE $\delta_{p1}(X,S)=m_{01}S_1^k$ under the GPC criterion.
		\end{corollary}
		Note that the following corollary provides improvements over the unrestricted BAEE $\delta_{0}(S)=d_{01}S_1^k$.
		\begin{corollary}
			Let $l(t)$ and $u(t)$ be as defined in above suppose $P\big[l(t)\leq d_{01} \leq u(t)\big]<1$ $\forall \ \underline{\gamma} \in \Theta_0$. Define for any $t$, $\varphi_1^*(t)=\max \left\{ l(t), \min\left\{ d_{01} , u(t) \right\} \right\}$. Then for any every $\underline{\gamma} \in \Theta_0$ the estimator $\delta_{\varphi_1^*}(X,S)=\varphi_1^*(T)$ is Pitman closest to $\sigma_1^k$ than the BAEE $\delta_{01}(X,S)=d_{01}S_1^k$ under the GPC criterion.
		\end{corollary}

		\subsection{Improve estimation of $\sigma_2^k$}
		Let $\delta_{\xi_1}(\underline{X},S)=\xi_1(W)S_2^k$ and $\delta_{\xi_2}(\underline{X},S)=\xi_2(W)S_2^k$ be two location equivariant estimators of $\theta$, where $\xi_1$ and $\xi_2$ are two positive real-valued functions defined on $\mathbb{R}$. Then, the GPC of $\delta_{\xi_1}(X,S)$ relative to $\delta_{\xi_2}(X,S)$ is given by
		\begin{align*}
			\mbox{GPC}(\delta_{\xi_1},\delta_{\xi_2};\eta)&= P_{\eta}\left[L\left(\frac{\xi_1(W)S_2^k}{\sigma_2^k}\right)  < L\left(\frac{\xi_2(W)S_2^k}{\sigma_2^k}\right)\right]\  + \frac{1}{2} P_{\eta}\left[L\left(\frac{\xi_1(W)S_2^k}{\sigma_2^k}\right)  = L\left(\frac{\xi_2(W)S_2^k}{\sigma_2^k}\right)\right]\\ 
			&=E^W\Big[P\big[L(\xi_1(W)V_2^k) < L(\xi_2(W)V_2^k)|W\big]\Big]\ + \frac{1}{2} E^W\Big[P\big[L(\xi_2(W)V_2^k) < L(\xi_2(W)V_2^k)|W\big]\Big]\\
		\end{align*}
		For any fixed positive $t$, we define 
		\begin{align*}
			r_{2,\eta}(\xi_1(w),\xi_2(w),w) = P_{\eta}\big[L(\xi_1(w)V_2^k) &< L(\xi_2(w)V_2^k)\big| W=w\big]\ + \frac{1}{2} P_{\eta}\big[L(\xi_1(w) V_2^k) = L(\xi_2(w)V_2^k)\big | W=w\big]
		\end{align*}
		For any fixed $\eta\leq1$ and $w>0$, let $m_\eta(w)$ denote the median of the distribution $V_2$ given $W=w$. The conditional distribution $V_2$ given $W=w$ is 
		$$g_{\eta}(v_2|w)\propto \mbox{Gamma}\left(p_1+p_2-2, \frac{1}{\left(1+\frac{w}{\eta}\right)}\right)$$
		Thus, 
		\begin{align*}
			\int_{0}^{m_\eta(w)}g_{\eta}(v_2|w) =\frac{1}{2}\int_{0}^{\infty} g_{\eta}(v_2|w).
		\end{align*}
		For any fixed $w>0$, by using the Lemma \ref{pt1}, we obtain $r_{2,\eta}\big(\xi_1(w),\xi_2(w),w\big) >1/2$ $\forall \ 0<\eta \leq 1$, if $\xi_2(w)<\xi_1(w)\leq m_\eta(w)$ $\forall \ \eta \geq 0$ or if $m_\eta(w)\leq\xi_1(w)<\xi_2(w)$ $\forall \ \eta \geq 0$. Moreover, for any fixed $w$, $r_{2,\eta}(\xi_1(w),\xi_2(w),w)=1/2$ $\forall \ 0<\eta \leq 1$.
		\begin{theorem}
			Let $\delta_{\xi_1}(X,S)=\xi_1(W)S_2^k$ be a scale equivariant estimator of $\sigma_2^k$. For any $w$, let $l(w)$ and $u(w)$ be function such that $l(w)\leq m_\eta(w)\leq u(w)$ $\forall \ \eta \leq 1$. At any fixed $w$, we define $\widetilde{\xi_1}(w)=\max\left\{l(w),\min\left\{\xi_1(w), u(w)\right\}\right\}$. Then under GPC criterion with a general loss function $L(t)$ the estimator $\delta_{\widetilde{\xi_1}}(X,S)= \widetilde{\xi_1}(W)S_2^k$ is Pitman closest to $\sigma_2^k$ than the estimator $\delta_{\xi_1}(X,S)=\xi_1(W)S_2^k$ $\forall \ \underline{\gamma} \in \Theta_0$, provided $P\big[l(t)\leq \xi_1(t)\leq u(t)\big]<1$ $\forall \ \underline{\gamma} \in \Theta_0$.
		\end{theorem}
		\noindent \textit{Proof:} The theorem can be proved by using same argument as in Theorem \ref{thpt1}

		\begin{corollary}
			Let $l(w)$ and $u(w)$ be as defined as above, suppose that $P\big[l(w)\leq m_{02} \leq u(w)\big]<1$ $\forall \ \underline{\gamma} \in \Theta_0$. For any fixed $w>0$ defined $\widetilde{\xi_1}(w)=\max \left\{ l(w), \min\left\{ m_{02},u(w)\right\}\right\}=\min\left\{m_{02},u(w)\right\}$. Then for any every $\underline{\gamma} \in \Theta_0$ the estimator $\delta_{\widetilde{\xi_1}}(X,S)=\widetilde{\xi_1}(W)S_2^k$ is Pitman closest to $\theta$ than the (Pitman closest afine equivariant estimator) PCAEE $\delta_{p}(X,S)=m_{02}S_2^k$ under the GPC criterion.
		\end{corollary}
		Note that the following corollary provides improvements over the unrestricted BAEE $\delta_{02}(S)=d_{02}S_2^k$.
		\begin{corollary}
			Let $l(w)$ and $u(w)$ be as defined in above suppose $P\big[l(w)\leq d_{02} \leq u(w)\big]<1$ $\forall \ \underline{\gamma} \in \Theta_0$. Define for any $w$, $\xi_1^*(w)=\max \left\{ l(w), \min\left\{ d_{02} , u(w) \right\} \right\}$. Then for any every $\underline{\gamma} \in \Theta_0$ the estimator $\delta_{\xi_1^*}(X,S)=\xi_1^*(W)$ is Pitman closest to $\sigma_2^k$ than the BAEE $\delta_{p2}(X,S)=d_{02}S_2^k$ under the GPC criterion.
		\end{corollary} 
		
		
		\section{A simulation study}\label{se5}
		In this section, we use the Monte Carlo technique for the simulation study. In the simulation study we compare the relative risk improvement (RRI) of the obtained estimators with respect to BAEE. The RRI for an estimator $\delta$ with the respect to an estimator $\delta_0$ is defined as 
		\begin{align*}
			\text{RRI}(\delta)=\frac{\mbox{Risk}(\delta_0)-\mbox{Risk}(\delta)}{\mbox{Risk}(\delta_0)}\times 100
		\end{align*}
		Here, we have consider the simulation study for $k=2$. For this analysis, a total of $90{,}000$ random samples were generated from two exponential distributions, $E(\mu_1, \sigma_1^2)$ and $E(\mu_2, \sigma_2^2)$, across various choices of $(\mu_1, \mu_2)$, $(\sigma_1, \sigma_2)$ and $(p_1,p_2)$. The RRI of the obtained improved estimators for $\sigma_1^2$ are shown in Figure \ref{fig1}, \ref{fig2} and \ref{fig3} under the loss function $L_1(t)$, $L_2(t)$ and $L_3(t)$. For estimating $\sigma_1^2$, the following key observations are obtained from Figure \ref{fig1}.
		\begin{enumerate}
			\item [(i)] The RRI of Stein-type estimator $\delta_{11}^1$, $\delta_{12}^1$, $\delta_{13}^1$, $\delta_{14}^1$ increase as $\eta$ increases. Moreover, the improvement region corresponding to $\delta_{11}^1$ is larger than those of $\delta_{12}^1$, $\delta_{13}^1$ and $\delta_{14}^1$ for all values of $\eta$.
			\item [(ii)] The RRI of the estimator $\delta_{\varphi_{01}}^1$ is non-monotonic in $\eta$. Specifically,it increases as $\eta$ moves from 0, attains its maximum at moderate values of $\eta$ (approximately $0.4\leq \eta 0.6$) and then declines as $\eta$ approaches to 1.
			\item [(iii)] The estimator $\delta_{11}^1$, $\delta_{12}^1$, $\delta_{13}^1$, $\delta_{14}^1$ have relatively lower risk improvement for moderate values of $\eta$; however, their performance increases significantly as $\eta$ approaches to 1. Note that, when $\eta$ near to 1, these Stein-type estimator surpass the performance of the Brewster and Zidek-type estimator. In addition, the estimator $\delta_{11}^1$ achieves the highest RRI near $\eta=1$.
			\item [(iv)] It is observed that the estimator $\delta_{\varphi_{01}}^1$ performs better than $\delta_{11}^1$, $\delta_{12}^1$, $\delta_{13}^1$, $\delta_{14}^1$ over the range $0.1\leq \eta \leq 0.7$(approximately), but its performance decreases relative to them when $\eta\geq 0.7$(approximately).
			\item [(v)] The improvement regions for all improved estimators shrink as the sample size increases or as the parameter values $(\mu_1,\mu_2)$ move away from the origin $(0,0)$.
		\end{enumerate}
		Similar patterns are observed under the entropy loss function $L_2(t)$, the symmetric loss function $L_3(t)$ and the linex loss function $L_4(t)$. It is observed that a closed form expression for the improved estimator under the loss function $L_4(t)$ can not be obtained for $K\geq2$ hence we are not considered it in the simulation study. The RRI of the improved estimators with respect to BAEE for $\sigma_2^2$ under the loss functions $L_1(t)$, $L_2(t)$ and $L_3(t)$ are presented in Figures \ref{fig4}, \ref{fig5}, and \ref{fig6} respectively. We now describe the observations for estimating $\sigma_2^2$ under the quadratic loss function $L_1(t)$, as shown in Figure \ref{fig4}.
		\begin{enumerate}
			\item[(i)] The RRIs of the estimators $\delta^1_{21}$ and $\delta_{D}^1$ increase with increasing $\eta$. Furthermore, $\delta_{D}^1$ has the largest improvement region for all values of $\eta$, compared to the estimator $\delta_{21}^1$.
			\item [(ii)] As we observed for $\sigma_1^2$, the RRI of the estimator $\delta_{\xi_{01}}^1$ is not monotone. It increases as $\eta$ moves away from 0, attains its maximum at moderate values of $\eta$ (approximately $0.6 \leq \eta \leq 0.7$), and then decreases as $\eta$ approaches 1.
			\item [(iii)] The estimators $\delta_{21}^1$ and $\delta_{D}^1$ show relatively smaller risk improvement for moderate values of $\eta$, but their performance increases significantly as $\eta$ approaches 1. It is observed that when $\eta$ is close to 1, these estimators outperform the Brewster and Zidek–type estimator $\delta_{\xi_{01}}^1$. Moreover, among them, $\delta_{D}^1$ achieves the highest RRI values near $\eta = 1$.
			\item [(iv)] Furthermore, it is observed that the estimator $\delta_{\xi_{01}}^1$ performs better than $\delta_{21}^1$ and $\delta_{D}^1$ in the range $0.1 \leq \eta \leq 0.7$ (approximately), but performs decline when $\eta \geq 0.7$ (approximately). In contrast, $\delta_{D}^1$ dominates both $\delta_{21}^1$ and $\delta^1_{\xi_{01}}$ for $\eta \geq 0.7$ (approximately).
		\end{enumerate}
	
	\begin{figure}[htbp]
		\centering
		\subfigure[$(p_1,p_2)=(4,5),(\mu_1,\mu_2)=(0,0.1)$]{
			\includegraphics[width=0.30\textwidth]{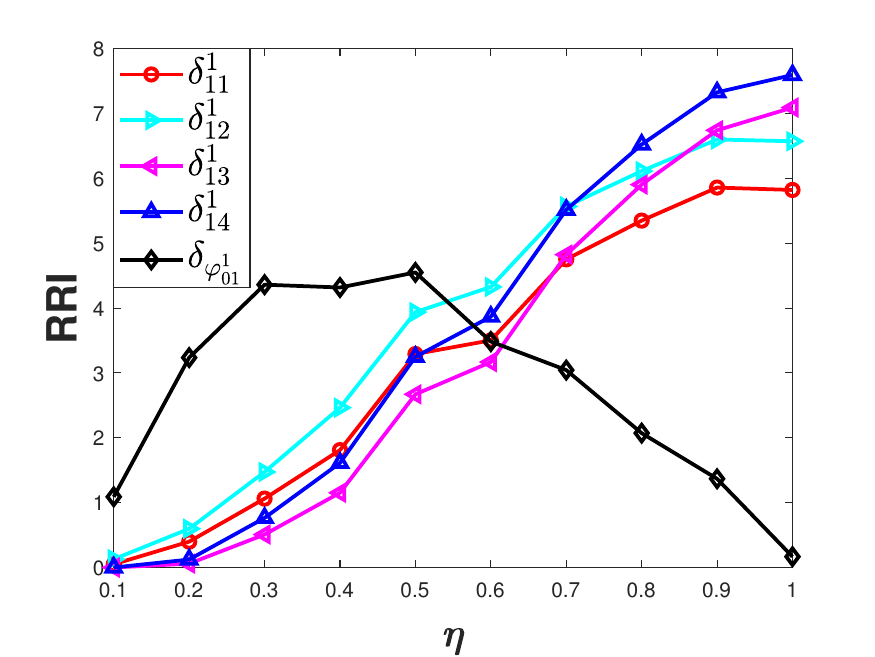}
		}
		\hfill
		\subfigure[$(p_1,p_2)=(8,6), (\mu_1,\mu_2)=(0.1,0.3)$]{
			\includegraphics[width=0.30\textwidth]{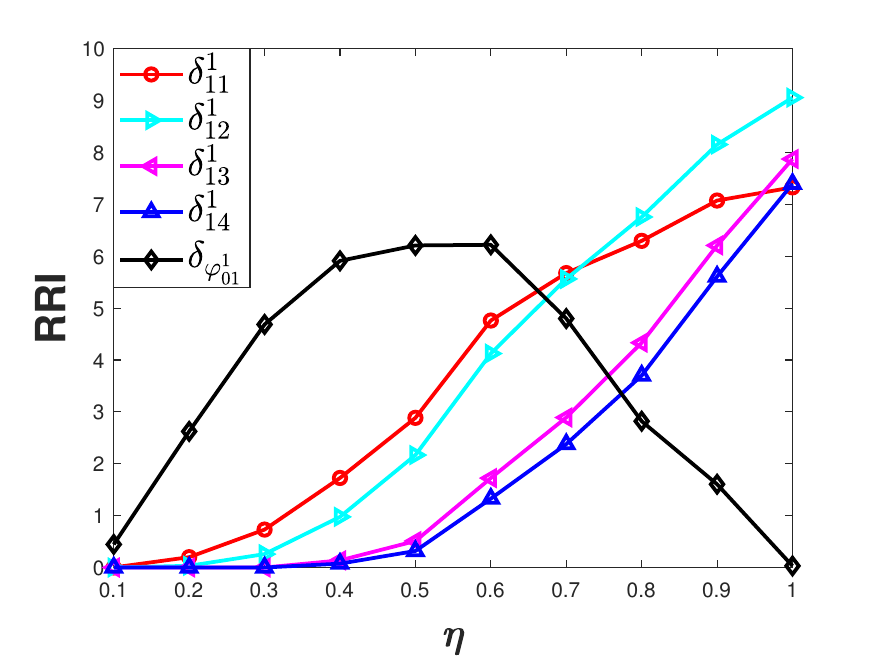}
		}
		\hfill
		\subfigure[$(p_1,p_2)=(8,12), (\mu_1,\mu_2)=(0.3,0.5)$]{
			\includegraphics[width=0.30\textwidth]{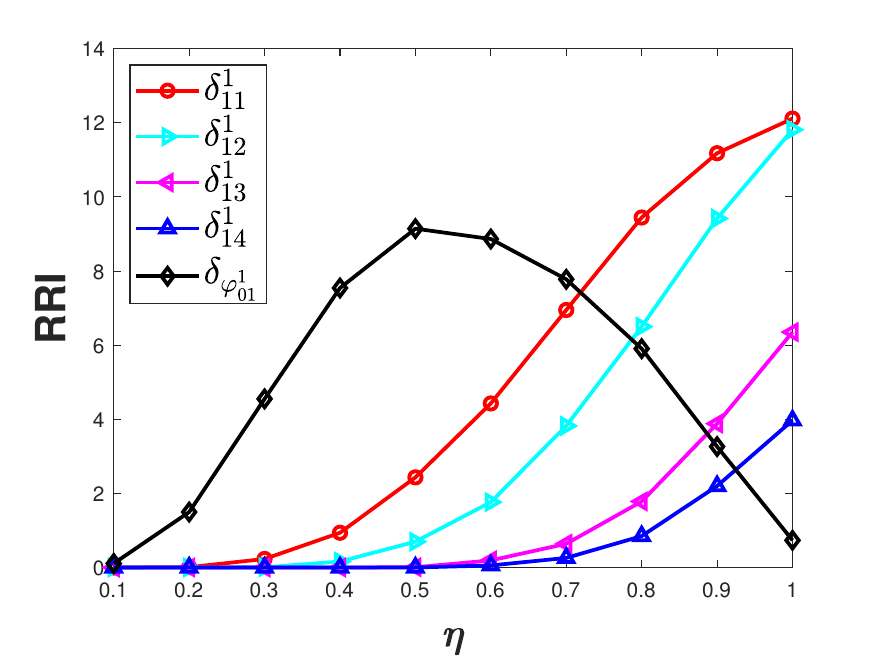}
		}
		
		\caption{RRI of estimators relative to BAEE $\sigma_1^2$ under $L_1(t)$}
		\label{fig1}
	\end{figure}
		\begin{figure}[htbp]
		\centering
		\subfigure[$(p_1,p_2)=(4,5),(\mu_1,\mu_2)=(0,0.1)$]{
			\includegraphics[width=0.30\textwidth]{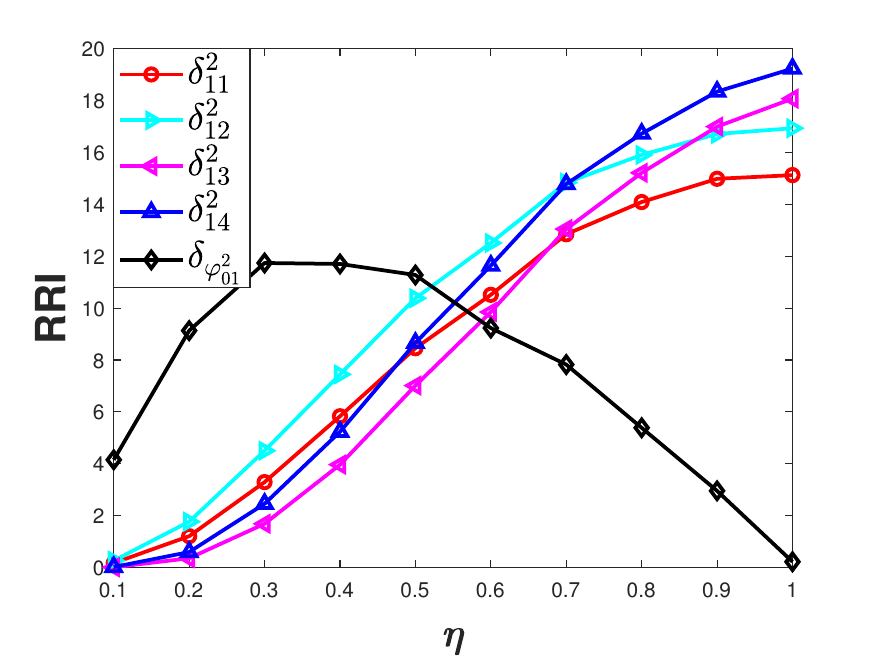}
		}
		\hfill
		\subfigure[$(p_1,p_2)=(8,6), (\mu_1,\mu_2)=(0.1,0.3)$]{
			\includegraphics[width=0.30\textwidth]{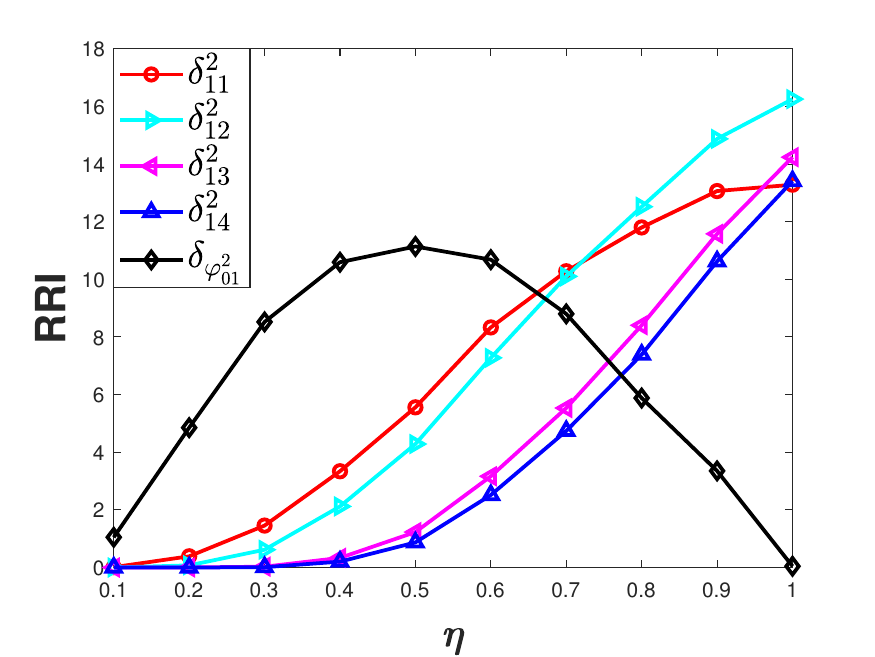}
		}
		\hfill
		\subfigure[$(p_1,p_2)=(8,12), (\mu_1,\mu_2)=(0.3,0.5)$]{
			\includegraphics[width=0.30\textwidth]{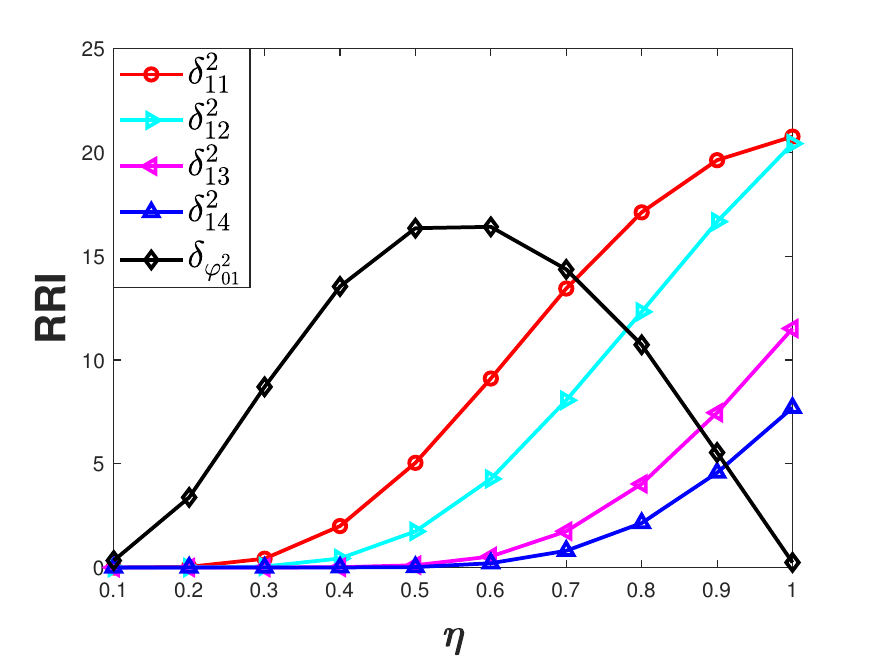}
		}
		
		\caption{RRI of estimators relative to BAEE $\sigma_1^2$ under $L_2(t)$}
		\label{fig2}
	\end{figure}
	\begin{figure}[htbp]
		\centering
		\subfigure[$(p_1,p_2)=(4,5),(\mu_1,\mu_2)=(0,0.1)$]{
			\includegraphics[width=0.30\textwidth]{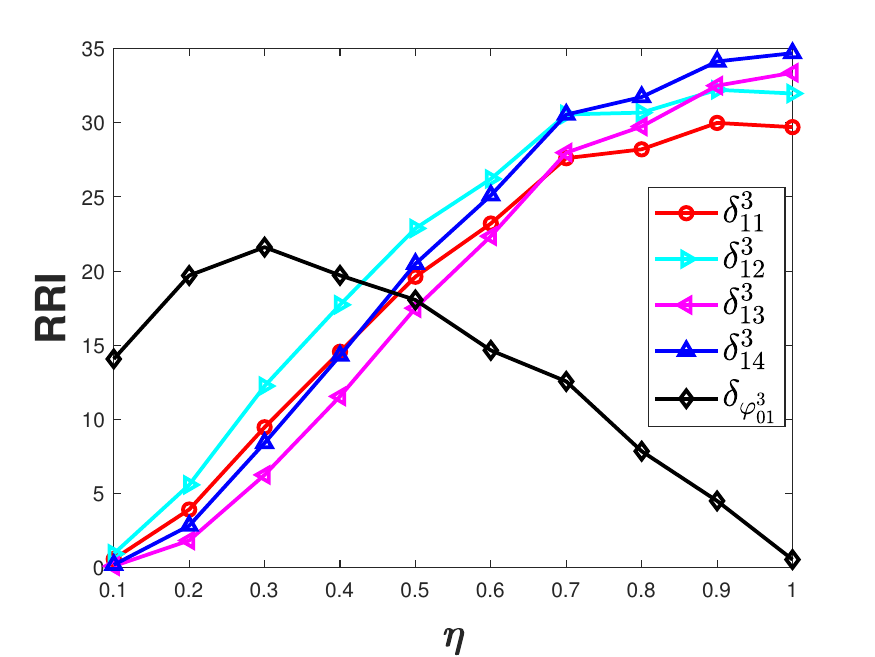}
		}
		\hfill
		\subfigure[$(p_1,p_2)=(8,6), (\mu_1,\mu_2)=(0.1,0.3)$]{
			\includegraphics[width=0.30\textwidth]{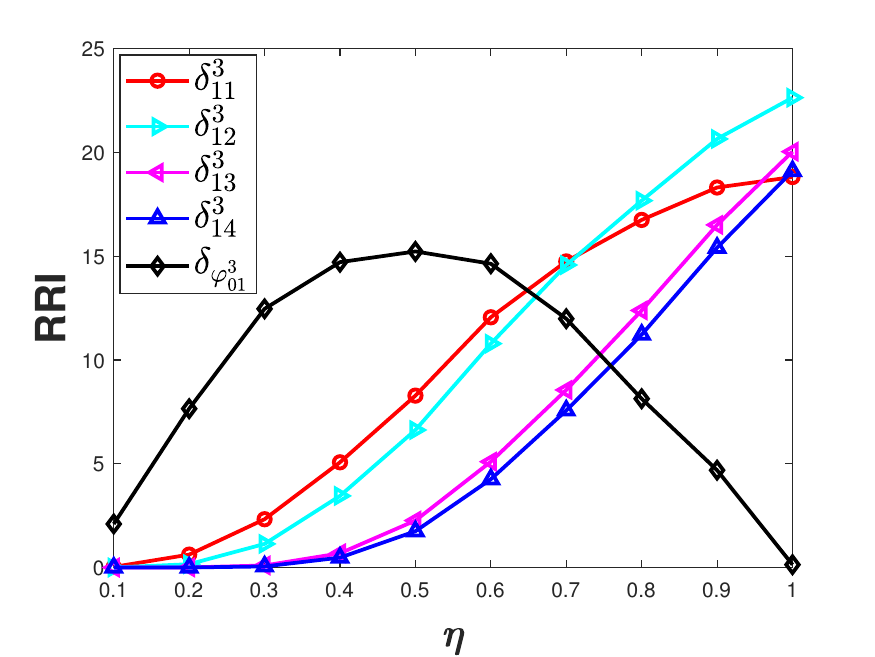}
		}
		\hfill
		\subfigure[$(p_1,p_2)=(8,12), (\mu_1,\mu_2)=(0.3,0.5)$]{
			\includegraphics[width=0.30\textwidth]{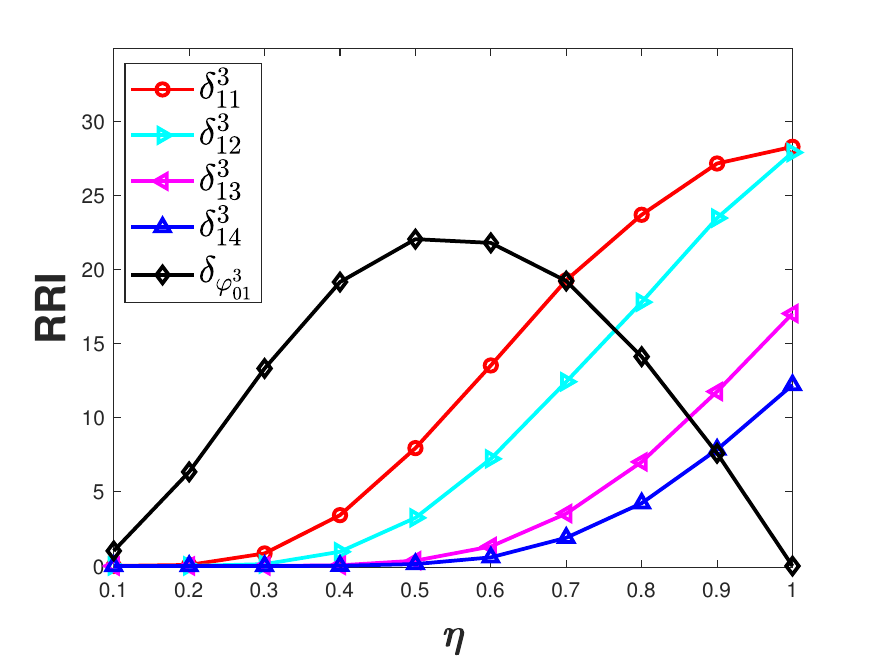}
		}
		
		\caption{RRI of estimators relative to BAEE $\sigma_1^2$ under $L_3(t)$}
		\label{fig3}
	\end{figure}
	\begin{figure}[htbp]
	\centering
	\subfigure[$(p_1,p_2)=(4,5),(\mu_1,\mu_2)=(0,0.1)$]{
		\includegraphics[width=0.30\textwidth]{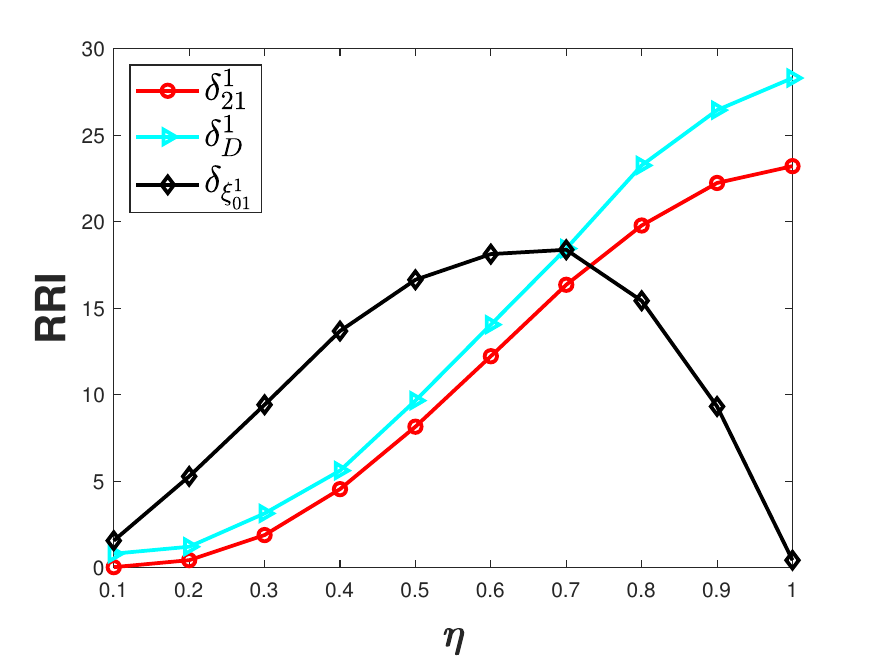}
	}
	\hfill
	\subfigure[$(p_1,p_2)=(8,6), (\mu_1,\mu_2)=(0.1,0.3)$]{
		\includegraphics[width=0.30\textwidth]{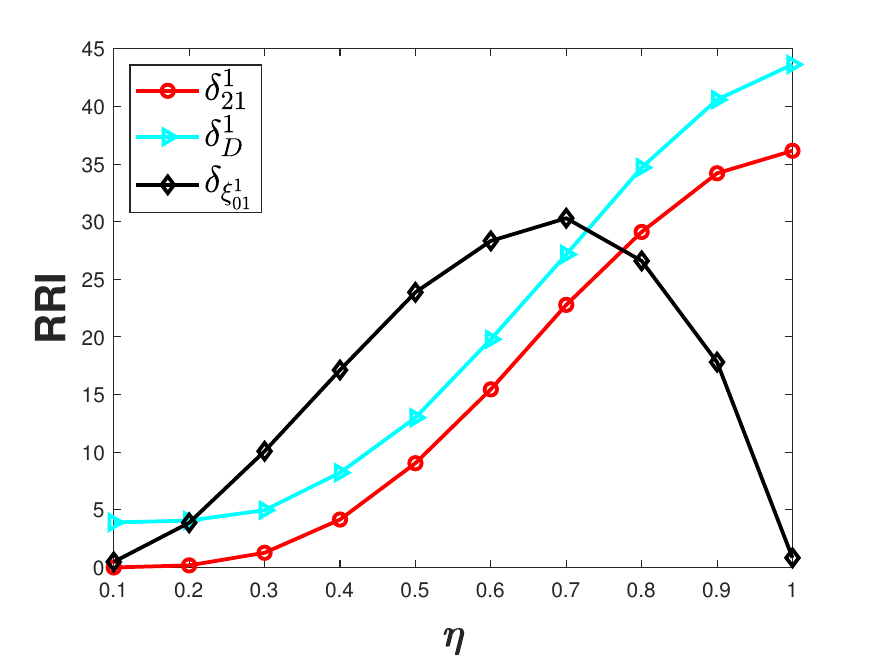}
	}
	\hfill
	\subfigure[$(p_1,p_2)=(8,12), (\mu_1,\mu_2)=(0.3,0.5)$]{
		\includegraphics[width=0.30\textwidth]{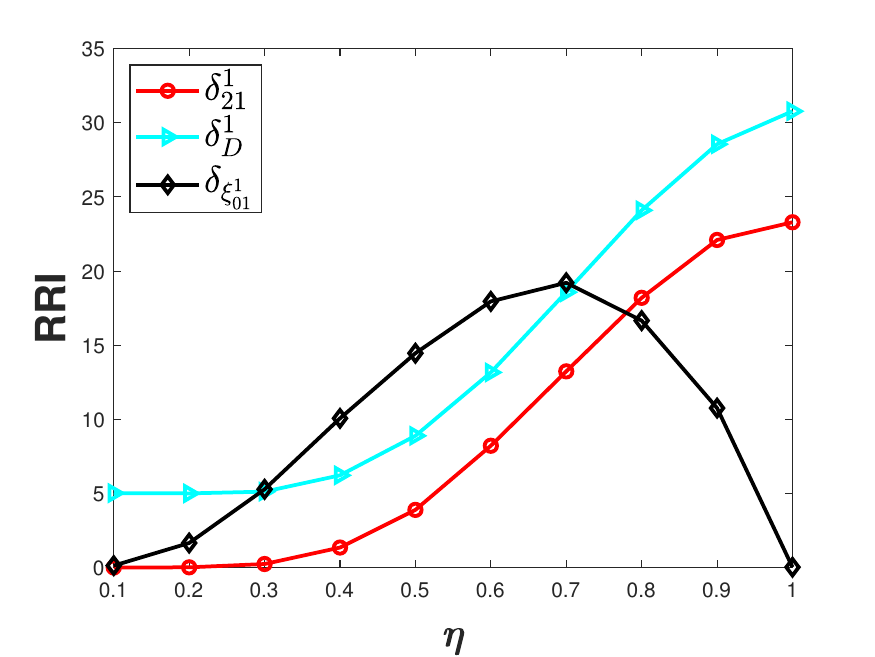}
	}
	
	\caption{RRI of estimators relative to BAEE $\sigma_2^2$ under $L_1(t)$}
	\label{fig4}
\end{figure}
		\begin{figure}[htbp]
		\centering
		\subfigure[$(p_1,p_2)=(4,5),(\mu_1,\mu_2)=(0,0.1)$]{
			\includegraphics[width=0.30\textwidth]{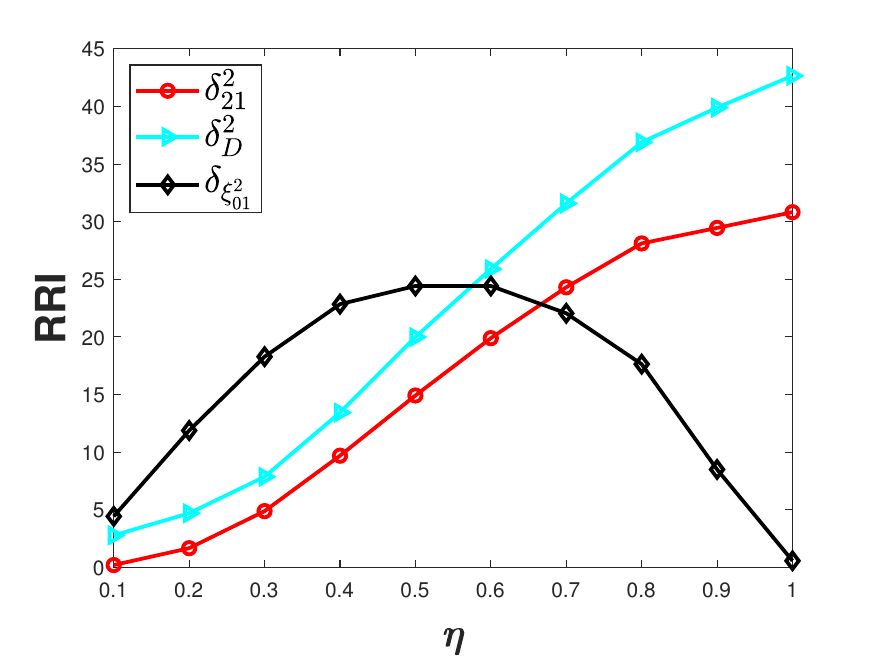}
		}
		\hfill
		\subfigure[$(p_1,p_2)=(8,6), (\mu_1,\mu_2)=(0.1,0.3)$]{
			\includegraphics[width=0.30\textwidth]{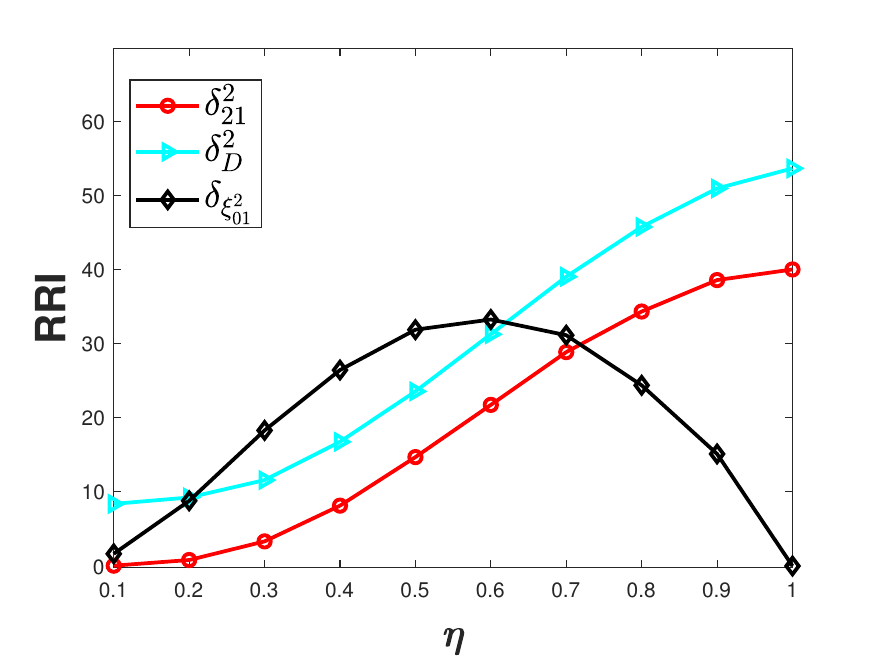}
		}
		\hfill
		\subfigure[$(p_1,p_2)=(8,12), (\mu_1,\mu_2)=(0.3,0.5)$]{
			\includegraphics[width=0.30\textwidth]{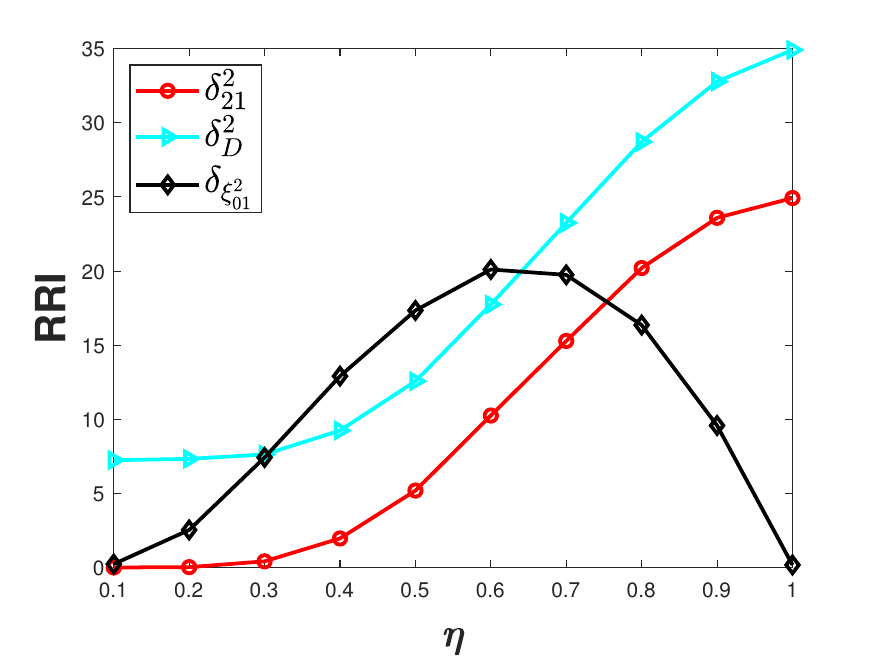}
		}
		
		\caption{RRI of estimators relative to BAEE $\sigma_2^2$ under $L_2(t)$}
		\label{fig5}
	\end{figure}
			\begin{figure}[htbp]
			\centering
			\subfigure[$(p_1,p_2)=(4,5),(\mu_1,\mu_2)=(0,0.1)$]{
				\includegraphics[width=0.30\textwidth]{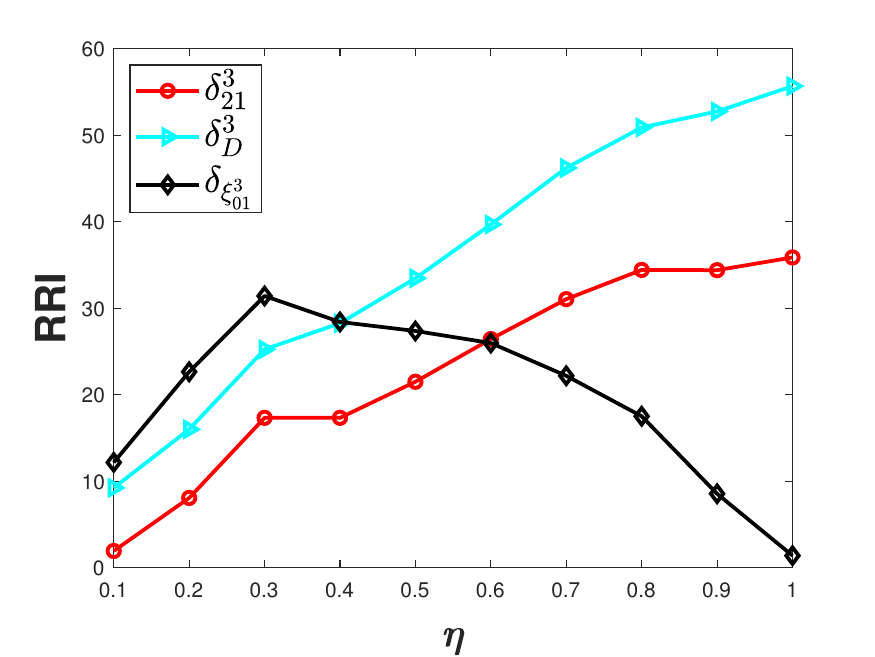}
			}
			\hfill
			\subfigure[$(p_1,p_2)=(8,6), (\mu_1,\mu_2)=(0.1,0.3)$]{
				\includegraphics[width=0.30\textwidth]{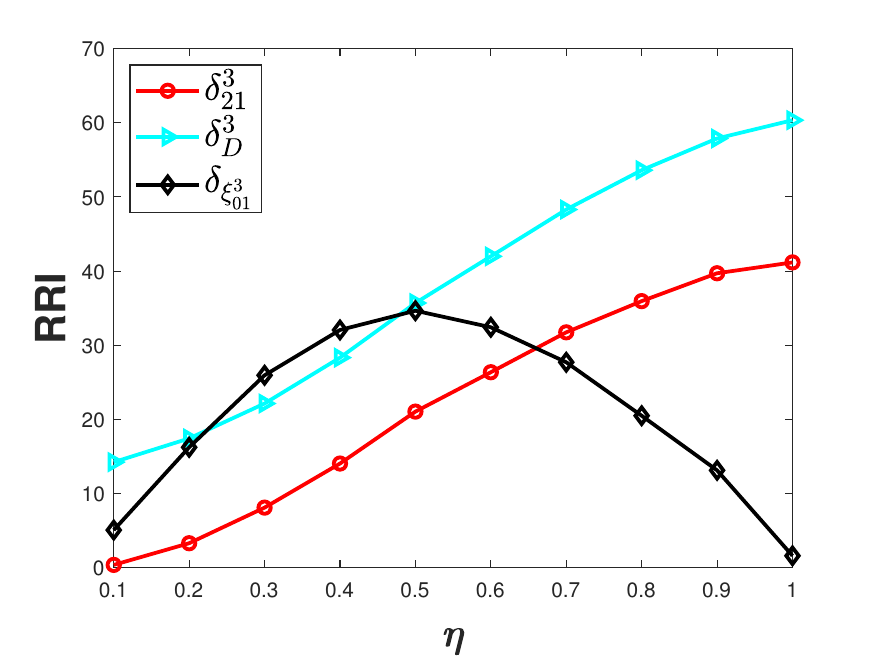}
			}
			\hfill
			\subfigure[$(p_1,p_2)=(8,12), (\mu_1,\mu_2)=(0.3,0.5)$]{
				\includegraphics[width=0.30\textwidth]{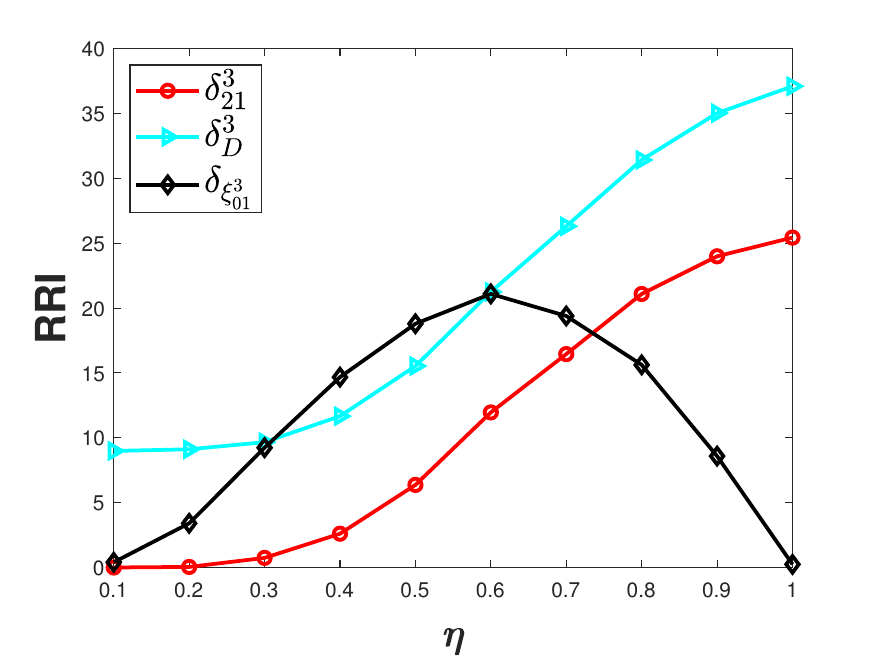}
			}
			
			\caption{RRI of estimators relative to BAEE $\sigma_2^2$ under $L_3(t)$}
			\label{fig6}
		\end{figure}
		\section{Data analysis}\label{se6}
		In this section we give an example of a real data set (see \cite{proschan1963theoretical}) that records the failure times (in hours) of the air conditioning systems for the Boeing $720$ aircraft "$7916$" and "$7907$". \textbf{Plane 7916} = $50,254,5,283,35,12$. \textbf{Plane 7907} = $194,15,41,29,33,181$. We check the given data sets using KS goodness of fit test. At the significance level $\alpha=0.05$, there is insufficient evidence to reject the hypothesis that the data for the plane 7916 and 7907 are from $Exp(5,0.0099)$ and $Exp(15,0.0149)$ distributions respectively. Based on the given data, we have computed the values of the statistics $X_1=5$, $X_2=15$, $S_1=609$, and $S_2=403$. These statistics are then used to obtain the estimated values of the improved estimators for $\sigma_1^2$ and $\sigma_2^2$ under the loss function $L_1(t)$, $L_2(t)$ and $L_3(t)$ respectively (see Tables \ref{table1} and \ref{table2}). For the linex loss function $L_4(t)$, the estimators do not have a closed form expression when $k=2$, therefore it is not included in the table of computed estimators. 		
		\begin{figure}[htbp]
			\centering
			\subfigure[Plane $7916$]{
				\includegraphics[width=0.8\textwidth]{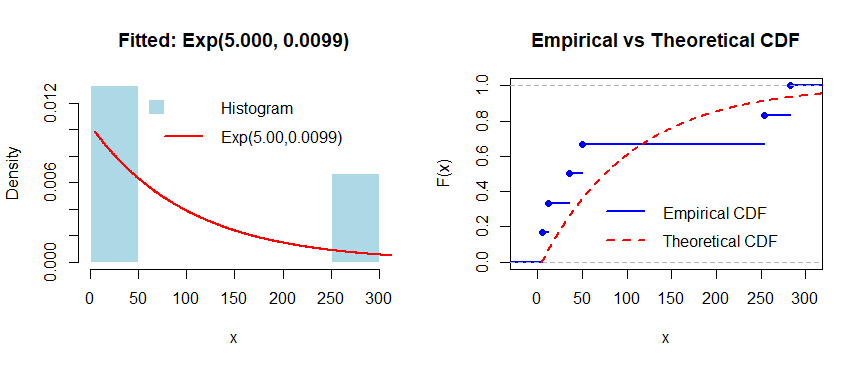}
			}
			\hfill
			\subfigure[Plane $7907$]{
				\includegraphics[width=0.8\textwidth]{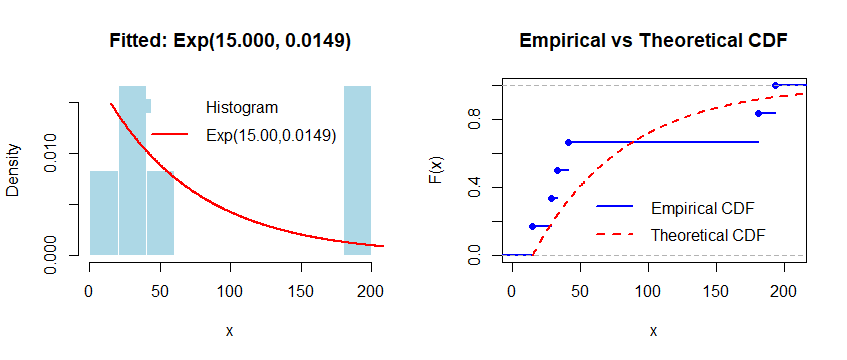}
			}
			\caption{Fitted exponential model with histogram and CDF comparison}\label{fig7}
		\end{figure}
		\begin{table}[h!]
			\centering
			\caption{Estimated values of the improved estimators of $\sigma_1^2$}\label{table1}
			\begin{tabular}{@{}ccccccc@{}}
				\toprule
				& $\delta_{01}$ & $\delta_{11}$ & $\delta_{12}$ &  $\delta_{13}$ & $\delta_{14}$ & $\delta_{\phi_{*}}$ \\
				\hhline{=======}
				$L_1(t)$  & $6.6229 \times$ $10^3$ & $6.5650 \times$ $10^3$ & $5.9657 \times$ $10^3$ & $6.6229 \times$ $10^3$  & $6.1020 \times$ $10^3$ & $4.7691 \times$ $10^3$\\
				
				$L_2(t)$  & $12.363 \times$ $10^3$ & $9.310 \times$ $10^3$ & $8.225 \times$ $10^3$ & $9.200 \times$ $10^3$  & $8.214 \times$ $10^3$ & $7.029 \times$ $10^3$\\
				
				$L_3(t)$   & $19.547 \times$ $10^3$ & $11.508 \times$ $10^3$ & $9.962 \times$ $10^3$ & $11.142 \times$ $10^3$  & $9.7782 \times$ $10^3$ & $8.840 \times$ $10^3$\\
				\bottomrule
			\end{tabular}
		\end{table}
		\begin{table}[h!]
			\centering
			\caption{Estimated values of the improved estimators of $\sigma_2^2$}\label{table2}
			\begin{tabular}{@{}ccccc@{}}
				\toprule
				& $\delta_{01}$ & $\delta_{21}$ & $\delta_{D}$ & $\delta_{\phi_{*}}$ \\
				\hhline{=====}
				$L_1(t)$  & $2.9002 \times$ $10^3$ & $6.5650\times$ $10^3$ & $6.0916\times$ $10^3$ & $9.0561\times$ $10^3$\\
				
				$L_2(t)$  & $5.414 \times$ $10^3$ & $9.310\times$ $10^3$ & $8.057\times$ $10^3$ & $14.642\times$ $10^3$\\
				
				$L_3(t)$  & $8.560\times$ $10^3$ & $11.508\times$ $10^3$ & $8.977\times$ $10^3$ & $20.842\times$ $10^3$\\
				\bottomrule
			\end{tabular}
		\end{table}
		\section{Concluding remarks} 
		In this manuscript, we study the estimation of positive powers of the ordered scale parameter for two exponential distributions under a general class of bowl-shaped, scale-invariant loss functions when $\sigma_1\leq \sigma_2$. For estimating $\sigma_i^k$, we derive several Stein-type improved estimators that perform better than the BAEE. Also we obtain a generalized Bayes estimator that turns out to be the same as the Brewster-Zidek-type improved estimator. We give the expression of the improved estimator for particular loss functions, such as quadratic, entropy, and symmetric loss functions. Moreover, improved estimators are developed under the generalized Pitman closeness criterion. A simulation study is carried out to compare the risk performance of the proposed estimator. In particular, for $k=2$, we examine the behaviour of different improved estimators of $\sigma_1^2$ and $\sigma_2^2$ under the above loss functions. The results indicate that the Brewster-Zidek type estimators perform better when $\eta\leq 0.74$ and when the parameters $(\mu_1,\mu_2)$ are close to the origin, whereas for $\eta\geq0.74$, the Stein-type estimators perform better. Finally, we use failure times of the air-conditioning systems of jet aircraft for data analysis, and compute improved estimates under the above loss functions. 
		
		\section*{Acknowledgment}
		The author is sincerely grateful to Dr Lakshmi Kanta Patra for many helpful discussions and for his guidance and support during the preparation of this research article.
		\section*{Data availability statement}
		Data used in this study are publicly available in the literature and cited in this manuscript.

	\end{document}